\documentclass[smallextended]{svjour3}       
\usepackage{amsmath}
\usepackage{amsfonts}
\usepackage{algorithm}
\usepackage{algorithmic}
\usepackage{mathtools}
\usepackage{url}
\usepackage{subfigure}
\usepackage[titletoc]{appendix}
\usepackage{xcolor}

\def\noprint#1{}

\newtheorem{assumption}{Assumption}

\newcommand{\lambdamin}{\lambda_{\mbox{\rm\scriptsize{min}}}}
\newcommand{\lambdamax}{\lambda_{\mbox{\rm\scriptsize{max}}}}
\newcommand{\ximin}{\xi_{\mbox{\rm\scriptsize{min}}}}
\newcommand{\ximax}{\xi_{\mbox{\rm\scriptsize{max}}}}

\newcommand{\R}{\mathbb{R}}

\newcommand{\eps}{\epsilon}
\newcommand{\epsg}{\epsilon_g}
\newcommand{\epsH}{\epsilon_H}
\newcommand{\flow}{f_{\mbox{\rm\scriptsize low}}}

\newcommand{\csol}{c_{sol}}
\newcommand{\jsol}{j_{sol}}
\newcommand{\cnc}{c_{nc}}
\newcommand{\jnc}{j_{nc}}
\DeclareMathOperator{\lspan}{span}

\newcommand{\cK}{{\cal K}}
\newcommand{\bH}{\bar{H}}
\newcommand{\hIcg}{\hat{J}}
\newcommand{\JB}{J}
\newcommand{\JBB}{\bar{J}}

\newcommand{\cO}{{\cal O}}
\newcommand{\tcO}{\tilde{\mathcal O}}

\DeclareRobustCommand{\Cmeo}{\mathcal{C}_{\mathrm{meo}}}
\DeclarePairedDelimiter{\ceil}{\lceil}{\rceil}

\newcommand{\refer}[1]{\textcolor{black}{#1}}

\def\tto{\;{\lower 1pt \hbox{$\rightarrow$}}\kern -10pt
           \hbox{\raise 2.8pt \hbox{$\rightarrow$}}\;}

\def\crcomment#1{\footnote{CR: #1}}

\begin{document}

\title{A Newton-CG Algorithm with Complexity Guarantees for Smooth Unconstrained Optimization %
  \thanks{Research supported by NSF Awards IIS-1447449, 1628384, 1634597,
and 1740707; AFOSR Award FA9550-13-1-0138; Subcontracts 3F-30222 and 8F-30039
from Argonne National Laboratory; and Award N660011824020 from the DARPA Lagrange Program.}}

\titlerunning{A Newton-CG Algorithm with Complexity Guarantees}

\author{Cl\'ement W. Royer \and 
    Michael O'Neill
    \and Stephen J. Wright}

\institute{C. W. Royer \at Wisconsin Institute of Discovery,
    University of Wisconsin, 330 N. Orchard St., Madison, WI
    53715.\\
    \email{croyer2@wisc.edu} 
    \and     
    M. O'Neill \at 
    Computer Sciences Department, University of
    Wisconsin, 1210 W. Dayton St., Madison, WI
    53706.\\
    \email{moneill@cs.wisc.edu} 
    \and
    S. J. Wright \at 
    Computer Sciences Department, University of
    Wisconsin, 1210 W. Dayton St., Madison, WI
    53706.\\
    \email{swright@cs.wisc.edu}.}

\date{Received: date / Accepted: date}

\maketitle

\begin{abstract}
We consider minimization of a smooth nonconvex objective function
using an iterative algorithm \refer{based on Newton's method and the 
linear conjugate gradient algorithm}, with explicit detection and use 
of negative curvature directions for the Hessian of the objective function. 
The algorithm tracks Newton-conjugate gradient procedures developed in the
1980s closely, but includes enhancements that allow worst-case
complexity results to be proved for convergence to points that satisfy
approximate first-order and second-order optimality conditions. The
complexity results match the best known results in the literature for
second-order methods.
\end{abstract}



\section{Introduction} 
\label{sec:intro}

We consider the unconstrained optimization problem
\begin{equation} \label{eq:f}
\min_{x \in \R^n} \, f(x),
\end{equation}
where $f:\R^n \to \R$ is a twice Lipschitz continuously differentiable
function that is nonconvex in general. We further assume that $f$ is
bounded below for all $x$, by some constant $\flow$. Although the
Hessian $\nabla^2 f(x)$ is well defined for such functions, we assume that 
full evaluation of this matrix is undesirable from a computational viewpoint, 
though we assume that Hessian-vector products of the form $\nabla^2 f(x) v$ 
can be computed with reasonable efficiency, for arbitrary vectors $v$,
\refer{as is often the case when $n$ is large.}

Unconstrained minimization of nonconvex smooth functions of many
variables is a much-studied paradigm in optimization. Approaches such
as limited-memory BFGS and nonlinear conjugate gradient are widely
used to tackle~\eqref{eq:f}, particularly in the case of large
dimension $n$.  Another popular approach, known as ``Newton-CG,''
applies the (linear) conjugate gradient (CG) method to the
second-order Taylor-series approximation of $f$ around the current
iterate $x_k$. Each iteration of CG requires computation of one
Hessian-vector product of the form $\nabla^2 f(x_k) v$. A trust-region
variant of Newton-CG, due to Steihaug~\cite{TSteihaug_1983},
terminates the CG iterations when sufficient accuracy is achieved in
minimizing the quadratic approximation, when a CG step leads outside
the trust region, or when negative curvature is encountered in
$\nabla^2 f(x_k)$.  A line-search variant presented
in~\cite{JNocedal_SJWright_2006} applies CG until some convergence
criterion is satisfied, or until negative curvature is encountered, in
which case the search direction reverts to the negative gradient.

Theoretical guarantees for Newton-CG algorithms have been provided,
e.g.  in~\cite{ARConn_NIMGould_PhLToint_2000,RSDembo_TSteihaug_1983,GFasano_SLucidi_2009,NIMGould_SLucidi_MRoma_PhLToint_2000,JNocedal_SJWright_2006}.  
Convergence analysis for such methods typically shows that
accumulation points are stationary, that is, they satisfy the
first-order optimality condition $\nabla f(x)=0$. Local linear or
superlinear convergence to a point satisfying second-order sufficient
conditions is sometimes also proved for Newton-CG methods. 
\refer{Although several Newton-type methods have been analyzed from 
a global complexity perspective~\cite{CCartis_NIMGould_PhLToint_2011a}, 
particularly in terms of outer iterations and derivative evaluations, 
bounds that explicitly account for the use of inexact Newton-CG 
techniques have received less attention in the optimization literature. 
Meanwhile, with the recent upsurge of interest in complexity, several new} 
algorithms have been proposed that have good global complexity 
guarantees. We review some such contributions in Section~\ref{sec:rw}. 
In most cases, these new methods depart significantly from those seen 
in the traditional optimization literature, and there are questions 
surrounding their practical appeal.

{\em Our aim in this paper is to develop a method that hews closely to
  the Newton-CG approach, but which comes equipped with certain
  safeguards and enhancements that allow worst-case complexity results
  to be proved.}  At each iteration, we use CG to solve a slightly
damped version of the Newton equations, monitoring the CG iterations
for evidence of indefiniteness in the Hessian. If the CG process
terminates with an approximation to the Newton step, we perform a
backtracking line search along this direction. Otherwise, we step
along a negative curvature direction for the Hessian, obtained either
from the CG procedure on the Newton equations, or via some auxiliary
computation (possibly another CG process). In either case, we can show
that significant decrease can be attained in $f$ at each iteration, at
reasonable computational cost (in terms of the number of gradient
evaluations or Hessian-vector products), allowing worst-case
complexity results to be proved.

The remainder of the paper is organized as follows. We position our
work within the existing literature in Section~\ref{sec:rw}. Our
algorithm is described in
Section~\ref{sec:algo}. Section~\ref{sec:wcc} contains the complexity
analysis, showing both a deterministic upper bound on the computation
required to attain approximate first-order conditions
(Section~\ref{subsec:wcc1}) and a high-probability upper bound on the
computation required to satisfy approximate second-order necessary
conditions (Section~\ref{subsec:wcc2}). Section~\ref{sec:discussion}
contains some conclusions and discussion.  Several technical results
and proofs related to CG are gathered in the Appendix.

\paragraph{Assumptions, Background, Notation.}

Our algorithm seeks a point that approximately satisfies second-order
necessary conditions for optimality, that is,
\begin{equation}\label{eq:2epson}
\| \nabla f (x)\| \le \epsg, \quad \lambda_{\min}(\nabla^2 f(x)) \ge
-\epsH,
\end{equation}
for specified small positive tolerances $\epsg$ and $\epsH$. (Here and
subsequently, $\| \cdot \|$ denotes the Euclidean norm, or its induced
norms on matrices.)
We make the following standard assumptions throughout.
\begin{assumption} \label{assum:compactlevelset}
The level set $\mathcal{L}_f(x_0) = \{x | f(x) \le f(x_0)\}$ is
compact.
\end{assumption}

\begin{assumption} \label{assum:fC22}
The function $f$ is twice uniformly Lipschitz continuously
differentiable on an open neighborhood of $\mathcal{L}_f(x_0)$
\refer{that includes the trial points generated by the algorithm. We
  denote by $L_H$ the Lipschitz constant for $\nabla^2 f$ on this
  neighborhood}.
\end{assumption}

\refer{Note that Assumption~\ref{assum:fC22} is made for simplicity of
  exposition; slightly weaker variants could be used at the expense of
  some complication in the analysis.}

Under these two assumptions, there exist scalars $\flow$, $U_g
> 0$, and $U_H > 0$ such that the following are satisfied for $x \in
\mathcal{L}_f(x_0)$:
\begin{equation}\label{eq:boundscompact}
	f(x) \geq \flow,\quad \|\nabla f(x)\| \leq U_g,
	\quad \|\nabla^2 f(x)\| \leq U_H.
\end{equation}
\refer{We observe that $U_H$ is a Lipschitz constant for the gradient.}

For any $x$ and $d$ such that Assumption~\ref{assum:fC22} is satisfied at $x$ 
and $x+d$, we have
\begin{equation} \label{eq:LH}
f(x+d) \le f(x) + \nabla f(x)^\top d + \frac12 d^\top \nabla^2 f(x) d
+ \frac{L_H}{6} \| d\|^3.
\end{equation}

Notationally, we use order notation ${\mathcal O}$ in the usual sense,
whereas $\tcO(\cdot)$ represents $\cO$ with logarithmic terms
\refer{omitted}.
We use such notation in bounding iteration count and computational
effort, and focus on the dependencies of such complexities on $\epsg$
and $\epsH$. \refer{(In one of our final results ---
  Corollary~\ref{coro:wcc2Hv} --- we also show explicitly the
  dependence on $n$ and $U_H$.)}

\section{Complexity in Nonconvex Optimization}
\label{sec:rw}

In recent years, many algorithms have been proposed for finding a
point that satisfies conditions \eqref{eq:2epson}, with iteration
complexity and computational complexity bounded in terms of $\epsg$
and $\epsH$.  We review several works most relevant to this paper
here, and relate their results to our contributions. For purposes of
computational complexity, we define the unit of computation to be one
Hessian-vector product {\em or } one gradient evaluation, implicitly
using the observation from computational / algorithmic differentiation
\cite{AGriewank_AWalther_2008} that these two operations differ in
cost only by a modest factor, independent of the dimension $n$.

Classical second-order convergent trust-region schemes
\cite{ARConn_NIMGould_PhLToint_2000} can be shown to satisfy
(\ref{eq:2epson}) after at most
$\mathcal{O}(\max\{\epsg^{-2}\epsH^{-1}, \epsH^{-3}\})$ iterations
\cite{CCartis_NIMGould_PLToint_2012}. For the class of second-order
algorithms (that is, algorithms which rely on second-order derivatives
and Newton-type steps) the best known iteration bound is
$\mathcal{O}(\max\{\epsg^{-3/2}, \epsH^{-3}\})$. This bound was first
established for a form of cubic regularization of Newton's method
\cite{YuNesterov_BTPolyak_2006}. Following this paper, numerous other
algorithms have also been proposed which match this bound, see for
example \cite{EGBirgin_JMMartinez_2017,CCartis_NIMGould_PhLToint_2011a,FECurtis_DPRobinson_MSamadi_2018,FECurtis_DPRobinson_MSamadi_2017,JMMartinez_MRaydan_2017}.

A recent trend in complexity analysis of these methods also accounts
for the computational cost of each iteration, thus yielding a bound on
the computational complexity. Two independent proposals, respectively
based on adapting accelerated gradient to the nonconvex setting
\cite{YCarmon_JCDuchi_OHinder_ASidford_2018a} and approximately
solving the cubic subproblem
\cite{NAgarwal_ZAllenZhu_BBullins_EHazan_TMa_2017}, require $\tcO
(\epsilon^{-7/4})$ operations (with high probability, showing
dependency only on $\epsilon$) to find a point $x$ that satisfies
\begin{equation} \label{eq:2on1eps}
\|\nabla f(x)\| \leq \epsilon \; \mbox{ and } \;
\lambdamin (\nabla^2 f(x)) \geq - \sqrt{U_H \epsilon}.
\end{equation}
The difference of a factor of $\epsilon^{-1/4}$ with the results
presented above arises from the cost of computing a negative curvature
direction of $\nabla^2 f(x_k)$ and/or the cost of solving a linear
system. The probabilistic nature of the bound is generally due to the
introduction of randomness in the curvature estimation process;
see~\cite{ZAllenZhu_YLi_2018,YXu_RJin_TYang_2018} for two recent
examples. A complexity bound of the same type was also
established for a variant of accelerated gradient free of negative
curvature computation, that regularly adds a random perturbation to
the iterate when the gradient norm is
small~\cite{CJin_PNetrapalli_MIJordan_2018}.

In an interesting followup
to~\cite{YCarmon_JCDuchi_OHinder_ASidford_2018a}, an algorithm based
on accelerated gradient with a nonconvexity monitor was
proposed~\cite{YCarmon_JCDuchi_OHinder_ASidford_2017b}. It requires at
most $\tcO ( \epsilon^{-7/4})$ iterations to satisfy
\eqref{eq:2on1eps} with high probability. However, if one is concerned
only with satisfying the gradient condition $\| \nabla f(x) \| \le
\eps$, the $\tcO(\eps^{-7/4})$ bound holds deterministically. Note
that this bound represents an improvement over the
$\mathcal{O}(\eps^{-2})$ of steepest descent and classical Newton's
method~\cite{CCartis_NIMGould_PhLToint_2010}.  The improvement is due
to a modification of the accelerated gradient paradigm that allows
for deterministic detection and exploitation of negative curvature
directions in regions of sufficient nonconvexity.

In a previous work \cite{CWRoyer_SJWright_2018}, two authors of the
current paper proposed a Newton-based algorithm in a line-search
framework which has an iteration complexity of
$\mathcal{O}(\max\{\epsg^{-3/2}, \epsH^{-3}\})$ when the subproblems
are solved exactly, and a computational complexity of
$\tcO\left(\epsilon^{-7/4}\right)$ Hessian-vector products and/or
gradient evaluations, when the subproblems are solved inexactly using
CG and the randomized Lanczos algorithm.  Compared to the accelerated
gradient methods, this approach aligns more closely with traditional
optimization practice, as described in Section~\ref{sec:intro}.

Building on~\cite{CWRoyer_SJWright_2018}, the current paper describes
a similar line-search framework with inexact Newton steps, but uses a
modified version of CG to solve the system of Newton equations,
without first checking for positive definiteness of the coefficient
matrix. The modification is based \refer{in part on a convexity
  monitoring device introduced in the accelerated gradient algorithms
  mentioned above.} An implicit cap is imposed on the number of CG
iterations that are used to solve the damped Newton system. We show
that once this cap is reached, either the damped Newton system has
been solved to sufficient accuracy or else a direction of
``sufficiently negative curvature'' has been identified for the
Hessian. (A single extra CG iteration may be needed to identify the
negative curvature direction, in much the same manner as in
\cite{YCarmon_JCDuchi_OHinder_ASidford_2017b} for accelerated
gradient.) In contrast to the previous work
\cite{CWRoyer_SJWright_2018}, no estimate of the smallest eigenvalue
of the Hessian is required prior to computing a Newton step. In
addition to removing potentially unnecessary computation, this
approach allows a deterministic result for first-order optimality to
be proved, as in \cite{YCarmon_JCDuchi_OHinder_ASidford_2017b}.


\refer{We are deliberate in our use of CG rather than accelerated
  gradient as the method of choice for minimizing the quadratic
  objective that arises in the damped Newton step. When applied to
  strongly convex quadratics, both approaches have similar asymptotic
  linear convergence rates that depend only on the extreme eigenvalues
  of the Hessian, and both can be analyzed using the same potential
  function~\cite{SKarimi_SAVavasis_2016} and viewed as two instances
  of an underlying ``idealized
  algorithm''~\cite{SKarimi_SAVavasis_2017}.}  However, CG has several
advantages: It has a rich convergence theory that depends on the full
spectrum of eigenvalues; it is fully adaptive, requiring no prior
estimates of the extreme eigenvalues; and its practical performance on
convex quadratics is superior. (See, for example,
\cite[Chapter~5]{JNocedal_SJWright_2006} for a description of these
properties.) Further, as we prove in this paper
(Section~\ref{subsec:algoccg}), CG can be adapted to detect
nonconvexity efficiently in a quadratic function, without altering its
core properties. We show in addition (see Section~\ref{subsec:algomeo}
and Appendix~\ref{app:lacg}) that by applying CG to a linear system
with a random right-hand side, we can find a direction of negative
curvature in an indefinite matrix efficiently, with the same iteration
complexity as the randomized Lanczos process of
\cite{JKuczynski_HWozniakowski_1992} used elsewhere.

The practical benefits of CG in large-scale optimization have long
been appreciated. We establish here that with suitable enhancements,
methods based on CG can also be equipped with good complexity
properties as well.

\section{Damped-Newton / Capped-CG Method with Negative Curvature Steps} 
\label{sec:algo}

We describe our algorithm in this section, \refer{starting with its
  two major components. The first component, described in
  Section~\ref{subsec:algoccg}, is a linear conjugate gradient
  procedure that is used to solve a slightly damped Newton
  system. This procedure includes enhancements to detect
  indefiniteness in the Hessian and to restrict the number of
  iterations. Because of this implicit bound on the number of
  iterations, we refer to it as ``Capped CG.''  The second component
  (see Section~\ref{subsec:algomeo}) is a ``minimum eigenvalue
  oracle'' that seeks a direction of negative curvature for a
  symmetric matrix, along with its corresponding vector. The main
  algorithm is described in Section~\ref{subsec:algodampednewt}.}

\subsection{Capped Conjugate Gradient}
\label{subsec:algoccg}

Conjugate Gradient (CG) is a widely used technique for solving linear
equations with symmetric positive definite coefficient matrices or,
equivalently, minimizing strongly convex quadratic functions.  We
devise a modified CG algorithm and apply it to a system of the form
$\bH y = -g$, where $\bH = H + 2\eps I$ is a damped version of the
symmetric matrix $H$, which is our notational proxy for the Hessian
$\nabla^2 f(x_k)$.

\begin{algorithm}[ht!]
\caption{Capped Conjugate Gradient}
\label{alg:ccg}
\begin{algorithmic}
\STATE \emph{Inputs:} Symmetric matrix $H \in \R^{n \times n}$; 
vector $g \ne 0$; damping parameter $\epsilon \in (0,1)$; desired relative 
accuracy \refer{$\zeta \in (0,1)$};
\STATE \refer{\emph{Optional input:} scalar $M$ (set to $0$ if not provided)};
\STATE \refer{\emph{Outputs:} d\_type, $d$;}
\STATE \refer{\emph{Secondary outputs:} final values of $M$, $\kappa$, 
$\hat{\zeta}$, $\tau$, and $T$;}
\STATE Set
\[
\bH:=H+2\eps I, \quad \kappa := \frac{M+2\eps}{\eps}, \quad \hat{\zeta} := \frac{\zeta}{3 \kappa},
\quad
\tau := \frac{\sqrt{\kappa}}{\sqrt{\kappa}+1}, \quad
T:=\frac{4\kappa^4}{(1-\sqrt{\tau})^2};
\]
\STATE $y_0 \leftarrow 0$, $r_0 \leftarrow g$, $p_0 \leftarrow -g$, $j \leftarrow 0$; 
\IF{$p_0^\top \bH p_0 < \eps \|p_0\|^2$}
\STATE Set $d=p_0$ and terminate with d\_type=NC;
\ELSIF{\refer{$\|H p_0\| > M \|p_0\|$}}
\STATE \refer{Set $M \leftarrow {\|H p_0\|}/{\|p_0\|}$ 
and update $\kappa,\hat{\zeta},\tau,T$ accordingly;}
\ENDIF
\WHILE{TRUE}
\STATE $\alpha_j \leftarrow {r_j^\top r_j}/{p_j^\top \bH p_j}$;
\COMMENT{Begin Standard CG Operations}
\STATE $y_{j+1} \leftarrow y_j+\alpha_j p_j$;
\STATE $r_{j+1} \leftarrow r_j + \alpha_j \bH p_j$;
\STATE $\beta_{j+1} \leftarrow {(r_{j+1}^\top r_{j+1})}/{(r_j^\top r_j)}$;
\STATE $p_{j+1} \leftarrow  -r_{j+1} + \beta_{j+1}p_j$;
\COMMENT{End Standard CG Operations}
\STATE $j \leftarrow  j+1$;
\IF{\refer{$\|H p_j\| > M \|p_j\|$}}
\STATE \refer{Set $M \leftarrow {\|H p_j\|}/{\|p_j\|}$ 
and update $\kappa,\hat{\zeta},\tau,T$ accordingly;}
\ENDIF
\IF{\refer{$\|H y_j\| > M \|y_j\|$}}
\STATE \refer{Set $M \leftarrow {\|H y_j\|}/{\|y_j\|}$ 
and update $\kappa,\hat{\zeta},\tau,T$ accordingly;}
\ENDIF
\IF{\refer{$\| H r_j\| > M \|r_j\|$}}
\STATE \refer{Set $M  \leftarrow {\| H r_j\|}/{\|r_j\|}$ 
and update $\kappa,\hat{\zeta},\tau,T$ accordingly;}
\ENDIF
\IF{$y_j^\top \bH y_j < \epsilon \|y_j\|^2$}
\STATE Set $d \leftarrow y_j$ and terminate with d\_type=NC;
\ELSIF{$\| r_j \| \le \hat{\zeta} \|r_0\|$}
\STATE Set $d \leftarrow y_j$ and terminate with d\_type=SOL;
\ELSIF{$p_j^\top \bH p_j < \eps \|p_j\|^2$}
\STATE Set $d \leftarrow p_j$ and terminate with d\_type=NC;
\ELSIF{$\|r_j\| >  \sqrt{T}\tau^{j/2} \|r_0\|$
}
\STATE Compute $\alpha_j,y_{j+1}$ as in the main loop above;
\STATE Find $i \in \{0,\dotsc,j-1\}$ such that
\begin{equation} \label{eq:weakcurvdir}
	\frac{(y_{j+1}-y_i)^\top \bH (y_{j+1}-y_i)}{\|y_{j+1}-y_i\|^2} \; < \; \eps;
\end{equation}
\STATE Set $d \leftarrow y_{j+1}-y_i$ and terminate with d\_type=NC;
\ENDIF
\ENDWHILE
\end{algorithmic}
\end{algorithm}


Algorithm~\ref{alg:ccg} presents our Capped CG procedure. The main
loop consists of classical CG iterations. When $\bH \succeq \eps I$,
Algorithm~\ref{alg:ccg} will generate the same iterates as a classical
conjugate gradient method \refer{applied to $\bH y=-g$, and terminate
  at an inexact solution of this linear system}.  When $\bH
\not\succeq \eps I$, the features added to Algorithm~\ref{alg:ccg}
cause a direction $d$ to be identified along which $d^\top \bH d <
\eps \|d\|^2$ or, equivalently, $d^\top H d < -\eps \|d\|^2$ --- a
direction of ``sufficiently negative curvature'' for $H$.  Directions
$d$ of this type are encountered most obviously when they arise as
iterates $y_j$ or search directions $p_j$ in the CG iterations. But
evidence of the situation $\bH \not\succeq \eps I$ can arise more
subtly, when the residual norms $\|r_j \|$ decrease more slowly than
we would expect if the eigenvalues of $\bH$ were bounded below by
$\eps$. Accordingly, Algorithm~\ref{alg:ccg} checks residual norms for
slow decrease, and if such behavior is detected, it uses a technique
based on one used for accelerated gradient methods in
\cite{YCarmon_JCDuchi_OHinder_ASidford_2017b} to recover a direction
$d$ such that $d^\top \bH d < \eps \|d\|^2$.

\refer{Algorithm~\ref{alg:ccg} may be called with an optional input
  $M$ that is meant to be an upper bound on $\|H\|$. Whether or not
  this parameter is supplied, it is updated so that at any point in
  the execution of the algorithm, $M$ is an upper bound on the maximum
  curvature of $H$ revealed to that point. Other parameters that
  depend on $M$ (namely, $\kappa$, $\hat\zeta$, $\tau$, and $T$) are
  updated whenever $M$ is updated.}

The following lemma justifies our use of the term ``capped'' in
connection with Algorithm~\ref{alg:ccg}. Regardless of whether the
condition $\bH \succeq \eps I$ is satisfied, the number of iterations
will not exceed a certain number $J(M,\eps,\zeta)$ that we
subsequently show to be $\tcO(\eps^{-1/2})$. (We write $J$ for
$J(M,\eps,\zeta)$ in some of the subsequent discussion, to avoid
clutter.)

\begin{lemma} \label{lemma:ccgits}
The number of iterations of Algorithm~\ref{alg:ccg} is bounded by
\[
\min\{n,J(M,\eps,\zeta)\},
\]
where $J=J(M,\eps,\zeta)$ is the smallest integer such that
$\sqrt{T}\tau^{J/2} \le \hat{\zeta}$ \refer{where $M$, $\hat\zeta$,
  $T$, and $\tau$ are the values returned by the algorithm}. \refer{If
  all iterates $y_i$ generated by the algorithm are stored, the number
  of matrix-vector multiplications required is bounded by
\begin{equation} \label{eq:pv8}
	\min\{n,J(M,\eps,\zeta)\}+1.
\end{equation}
If the iterates $y_i$ must be regenerated in order to define the
direction $d$ returned after \eqref{eq:weakcurvdir}, this bound
becomes $2\min\{n,J(M,\eps,\zeta)\}+1$.}
\end{lemma}
\begin{proof}
  If the full $n$ iterations are performed, without any of the
  termination conditions being tripped, the standard properties of CG
  (see Appendix~\ref{app:cgresults}) ensure that the final residual
  $r_n$ is zero, so that the condition $\|r_n \| \le \hat\zeta
  \|r_0\|$ is satisfied, and termination occurs.

  \refer{Since no more than $n$ iterations are performed, the upper
    bound $M$ is updated at most a finite number of times, so the
    quantity $J$ is well defined.}

  Supposing that $J<n$, we note from the definition of $J$ that
  $\sqrt{T}\tau^{J/2} \|r_0 \| \le \hat\zeta \|r_0\|$. Thus at least
  one of the following two conditions must be satisfied at iteration
  $J$: $\|r_{J} \| \le \hat\zeta \|r_0\|$ or $\|r_{J} \| >
  \sqrt{T}\tau^{J/2} \|r_0\|$. In either case, termination will occur
  at iteration $J$, unless it has occurred already at a previous
  iteration.

  To derive \eqref{eq:pv8}, note that the main workload at each
  iteration $j$ is computation of a single matrix-vector product
  \refer{$H p_j$ after the increment of $j$ (since 
    matrix-vector products involving the matrices
    $H$ and $\bH$ and the vectors $y_j$ and $r_j$ can be computed in
    terms of this vector, in an additional $O(n)$ operations).}
  \refer{(The ``$+1$'' in \eqref{eq:pv8} accounts for the initial
    matrix-vector multiplication $H p_0$ performed prior to entering 
    the loop.)}

  If we do not store additional information, we need to regenerate the
  information needed to compute the direction $d$ satisfying
  \eqref{eq:weakcurvdir} by re-running the iterations of CG, possibly
  up to the \refer{second-to-last} iteration. This fact accounts for
  the additional cost of $\min\{n,J(M,\eps,\zeta)\}$ in the no-storage
  case. \footnote{\refer{Interestingly, as we show in
      Appendix~\ref{app:cgresults}, the ratios on the left-hand side
      of \eqref{eq:weakcurvdir} can be calculated without knowledge of
      $y_i$ for $i=0,1,\dotsc,j-1$, provided that we store the scalar
      quantities $\alpha_i$ and $\|r_i\|^2$ for $i=0,1,\dotsc,j-1$.}}
\end{proof}

\refer{Note that $J(M,\epsilon,\zeta)$ is an increasing function of
  $M$, since $\hat{\zeta}$ is a decreasing function of $M$, while $T$
  and $\tau$ (and thus $\sqrt{T}\tau^{j/2}$) are increasing in $M$. If
  $U_H$ is known in advance, we can call Algorithm~\ref{alg:ccg} with
  $M = U_H$ and use $J(U_H,\epsilon,\zeta)$ as the bound. Alternately,
  we can call Algorithm~\ref{alg:ccg} with $M=0$ and let it adjust $M$
  as needed during the computation. Since the final value of $M$ will
  be at most $U_H$, and since $J(M,\epsilon,\zeta)$ is an increasing
  function of $M$, the quantity $J(U_H,\epsilon,\zeta)$ provides the
  upper bound on the number of iterations in this case too.}

We can estimate $J$ by taking logs in its definition, as follows:
\[
J \le \frac{2\ln(\hat{\zeta}/\sqrt{T})}{\ln(\tau)} =
\frac{\ln(\hat{\zeta}^2/T)}
     {\ln\left(\frac{\sqrt{\kappa}}{\sqrt{\kappa}+1}\right)} =
     \frac{\ln(T/\hat{\zeta}^2)}{\ln(1+{1}/{\sqrt{\kappa}})} \le
     \left(\sqrt{\kappa}+\frac{1}{2}\right)
     \ln\left(\frac{T}{\hat{\zeta}^2}\right),
\]
where we used $\ln(1+\tfrac{1}{t}) \ge \tfrac{1}{t+1/2}$ to obtain the
latest inequality. By
replacing $T,\tau,\hat{\zeta},\kappa$ by their
definitions in Algorithm~\ref{alg:ccg}, and using
$
\frac{1}{1-\sqrt{\tau}} = \frac{1+\sqrt{\tau}}{1-\tau} \le
\frac{2}{1-\tau},
$
we obtain
\begin{align} \nonumber
J(M,\eps,\zeta) &  \le
\refer{\min\left\{n,\left\lceil \left(\sqrt{\kappa}+\frac{1}{2}\right)
\ln\left(\frac{144 \left(\sqrt{\kappa}+1\right)^2
  \kappa^6} {\zeta^2}\right) \right\rceil \right\}} \\
\label{eq:HvcountcappedCG}
 & = \min \left\{ n , \tcO(\eps^{-1/2}) \right\}.
\end{align}


\subsection{Minimum Eigenvalue Oracle} 
\label{subsec:algomeo}

A minimum eigenvalue oracle is needed in the main algorithm to either
return a direction of ``sufficient negative curvature'' in a given
symmetric matrix, or else return a certificate that the matrix is
almost positive definite. This oracle is stated as
Procedure~\ref{alg:meo}.

\floatname{algorithm}{Procedure}

\begin{algorithm}[ht!]
\caption{Minimum Eigenvalue Oracle}
\label{alg:meo}
\begin{algorithmic}
  \STATE \emph{Inputs:} Symmetric matrix $H \in \R^{n \times n}$,
  \refer{tolerance $\epsilon>0$};
  \STATE \refer{\emph{Optional input:} Upper bound on Hessian norm 
  $M$;}
  \STATE \emph{Outputs:} An estimate $\lambda$ of $\lambdamin(H)$
  such that $\lambda \le - \epsilon/2$, and vector $v$
  with $\|v\|=1$ such that $v^\top H v =\lambda$ OR
  a certificate that $\lambdamin(H) \ge -\epsilon$. In the latter case, when
  the certificate is output, it is false with probability $\delta$ 
  \refer{for some $\delta \in [0,1)$}.
\end{algorithmic}
\end{algorithm}

To implement this oracle, we can use any procedure that finds the
smallest eigenvalue of $H$ to an absolute precision of $\epsilon/2$
\refer{with probability at least $1-\delta$. This probabilistic 
property encompasses both deterministic and randomized instances of 
Procedure~\ref{alg:meo}. In Section~\ref{subsec:wcc2}, we will establish 
complexity results under this general setting, and analyze the 
impact of the threshold $\delta$}.  Several possibilities
for implementing Procedure~\ref{alg:meo} have been proposed in the
literature, with various guarantees. An exact\refer{, deterministic} computation of the
minimum eigenvalue and eigenvector (through a full Hessian evaluation
and factorization) would be a valid choice for Procedure~\ref{alg:meo}
\refer{(with $\delta=0$ in that case)}, but is unsuited to our setting
in which Hessian-vector products and vector operations are the
fundamental \refer{operations}.
Strategies that require only gradient
evaluations~\cite{ZAllenZhu_YLi_2018,YXu_RJin_TYang_2018} may offer
similar guarantees to those discussed below.

We focus on two \refer{inexact, randomized} approaches for implementing 
Procedure~\ref{alg:meo}.
The first is the Lanczos method, which finds the smallest eigenvalue
of the restriction of a given symmetric matrix to a Krylov subspace
based on some initial vector. When the starting vector is chosen
randomly, the dimension of the Krylov subspace increases by one at
each Lanczos iteration, with high probability (see
Appendix~\ref{app:lacg} and \cite{JKuczynski_HWozniakowski_1992}).
To the best of our knowledge, \cite{YCarmon_JCDuchi_OHinder_ASidford_2018a} 
was the first paper \refer{to propose a complexity analysis based on 
the use of randomized Lanczos for detecting negative curvature}.  
The key result is the following.


\begin{lemma} \label{lemma:randlanczos}
\refer{Suppose that the Lanczos method is used to estimate the
  smallest eigenvalue of $H$ starting with a random vector uniformly
  generated on the unit sphere, where $\|H\| \le M$. For any $\delta
  \in [0,1)$, this approach finds the smallest eigenvalue of $H$ to an
  absolute precision of $\eps/2$, together with a corresponding
  direction $v$, in at most
\begin{equation} \label{eq:rlanc}
	\min \left\{n, 1+\ceil*{\frac12 \ln (2.75 n/\delta^2)
	\sqrt{\frac{M}{\epsilon}}} \right\} \quad \mbox{iterations},
\end{equation}	
with probability at least $1-\delta$.}
\end{lemma}
\begin{proof}
  If $\frac{\epsilon}{4M} \ge 1$, we have
  $-\tfrac{\epsilon}{4}I \prec -M I\preceq H  \preceq M I \prec \tfrac{\epsilon}{4}I$. 
  Therefore, letting $b$ be the (unit norm)
  random start of the Lanczos method, we obtain
  \[
  	b^\top H b \le M < \frac{\epsilon}{4} =-\frac{\epsilon}{4}+\frac{\epsilon}{2} < 
  	 -M + \frac{\epsilon}{2} \le \lambdamin(H) + \frac{\epsilon}{2},
  \]
  thus the desired conclusion holds at the initial point. 
  
  We now suppose that $\frac{\epsilon}{4M} \in (0,1)$.
  By setting $\bar\eps = \frac{\epsilon}{4M}$ in
  Lemma~\ref{lem:KW4.2}, we have that when $k$ is at least the
  quantity in \eqref{eq:rlanc}, the estimate $\ximin(H,b,k)$ of the
  smallest eigenvalue after $k$ iterations of Lanczos applied to $H$
  starting from vector $b$ satisfies the following bound, with
  probability at least $1-\delta$:
  \[
  \ximin(H,b,k) - \lambdamin(H) \le \bar\eps
  (\lambdamax(H)-\lambdamin(H)) \le \frac{\eps}{2} \frac{\lambdamax(H)
    - \lambdamin(H)}{2M} \le \frac{\eps}{2},
  \]
  as required.
\end{proof}



Procedure~\ref{alg:meo} can be
implemented by outputting the approximate eigenvalue $\lambda$ for
$H$, determined by the randomized Lanczos process, along with the
corresponding direction $v$, provided that $\lambda \le -\eps/2$. When
$\lambda>-\eps/2$, Procedure~\ref{alg:meo} returns the certificate
that $\lambdamin(H) \ge -\eps$\refer{, which is correct with 
probability at least $1-\delta$}.

The second approach to implementing Procedure~\ref{alg:meo} is to
apply the classical CG algorithm to solve a linear system in which the
coefficient matrix is a shifted version of the matrix $H$ and the
right-hand side is random. This procedure has essentially identical
performance to Lanczos in terms of the number of iterations required
to detect the required direction of sufficiently negative curvature,
as the following theorem shows.
\begin{theorem} \label{theo:randcg}
Suppose that Procedure~\ref{alg:meo} consists in applying the standard
CG algorithm (see Appendix~\ref{app:cgresults}) to the
linear system
\[
\left(H + \tfrac12 \epsilon I \right) d =  b,
\]
where $b$ is chosen randomly from a uniform distribution over the unit
sphere. \refer{Let $M$ satisfying $\|H\| \le M$ and $\delta \in (0,1)$
  be given.} If $\lambdamin(H) < -\eps$, then with probability at
least $1-\delta$, CG will yield a direction $v$ satisfying the
conditions of Procedure~\ref{alg:meo} in a number of iterations
bounded above by~\eqref{eq:rlanc}. Conversely, if CG runs for this
number of iterations without encountering a direction of negative
curvature for $H + \tfrac12 \epsilon I$, then $\lambdamin(H) \ge
-\eps$ with probability at least $1-\delta$.
\end{theorem}

We prove this result, and give some additional details of the CG
implementation, in Appendices~\ref{app:cgresults}
and~\ref{app:lacg}. \refer{We also present in
  Appendix~\ref{subapp:lacgwithoutM} a variant of the
  randomized-Lanczos implementation of Procedure~\ref{alg:meo} that
  does not require prior knowledge of the bound $M$ such that $\|H\|
  \le M$. In this variant, $M$ itself is also estimated via randomized
  Lanczos, and the number of iterations required does not different
  significantly from \eqref{eq:rlanc}. It follows from this result,
  together with our observation above that $M$ can also be obtained
  adaptively inside Algorithm~\ref{alg:ccg}, that knowledge of the
  bound on $\| \nabla^2 f(x) \|$ is not needed at all in implementing
  our method.}

\subsection{Damped Newton-CG}
\label{subsec:algodampednewt}

Algorithm~\ref{alg:dncg} presents our method for finding a point that
satisfies \eqref{eq:2epson}. It uses two kinds of search
directions. Negative curvature directions (that are also first-order
descent steps) are used when they are either encountered in the course
of applying the Capped CG method (Algorithm~\ref{alg:ccg}) to the
damped Newton equations, or found explicitly by application of
Procedure~\ref{alg:meo}. The second type of step is an inexact damped
Newton step, which is the other possible outcome of
Algorithm~\ref{alg:ccg}. For both types of steps, a backtracking line
search is used to identify a new iterate that satisfies a sufficient
decrease condition, that depends on the cubic norm of the step. Such a
criterion is instrumental in establishing optimal complexity
guarantees in second-order
methods~\cite{EGBirgin_JMMartinez_2017,FECurtis_DPRobinson_MSamadi_2018,FECurtis_DPRobinson_MSamadi_2017,CWRoyer_SJWright_2018}.

\floatname{algorithm}{Algorithm}

\begin{algorithm}[ht!]
\caption{Damped Newton-CG}
\label{alg:dncg}
\begin{algorithmic}
\STATE \emph{Inputs:} Tolerances  $\epsg>0$, $\epsH>0$; backtracking 
parameter $\theta \in (0,1)$; starting point $x_0$; 
\refer{accuracy parameter $\zeta \in (0,1)$;}
\STATE \refer{\emph{Optional input:} Upper bound  $M > 0$ on Hessian 
norm};
\FOR{$k=0,1,2,\dotsc$}
\IF{$\|\nabla f(x_k)\| > \epsg$}
\STATE Call Algorithm~\ref{alg:ccg} with $H = \nabla^2 f(x_k)$, $\epsilon=\epsH$, 
$g=\nabla f(x_k)$, accuracy parameter $\zeta$ and $M$ if provided, 
to obtain outputs $d$, d\_type; 
\IF{d\_type=NC}
\STATE $d_k \leftarrow -\mathrm{sgn}(d^\top g) \frac{|d^\top \nabla^2 f(x_k) d|}{\|d\|^2} 
\frac{d}{\|d\|}$;
\ELSE[\refer{d\_type=SOL}]
\STATE $d_k \leftarrow d$;
\ENDIF
\STATE Go to \textbf{Line Search};
\ELSE
\STATE Call Procedure~\ref{alg:meo} with $H = \nabla^2 f(x_k)$,
$\epsilon=\epsH$ and $M$ if provided; 
\IF{Procedure~\ref{alg:meo} certifies that $\lambdamin(\nabla^2 f(x_k)) \ge -\epsH$}
\STATE
Terminate; 
\ELSE[\refer{direction of sufficient negative curvature found}]
\STATE 
$d_k \leftarrow -\mathrm{sgn}(v^\top g) \frac{|v^\top \nabla^2 f(x_k) v|}{\|v\|^2} v$
(where $v$ is the output from Procedure~\ref{alg:meo})
and go to \textbf{Line Search};
\ENDIF
\ENDIF
\STATE \textbf{Line Search}: Compute a step length $\alpha_k=\theta^{j_k}$, 
where $j_k$ is the smallest nonnegative integer such that
\begin{equation} \label{eq:lsdecreasedamped}
	f(x_k + \alpha_k d_k) < f(x_k) - \frac{\eta}{6}\alpha_k^3 \|d_k\|^{3};
\end{equation}
\STATE $x_{k+1} \leftarrow x_k+\alpha_k d_k$;
\ENDFOR
\end{algorithmic}
\end{algorithm}

In its deployment of two types of search directions, our method is
similar to Steihaug's trust-region Newton-CG method \cite{TSteihaug_1983},
which applies CG (starting from a zero initial guess) to solve the
Newton equations but, if it encounters a negative curvature direction
during CG, steps along that direction to the trust-region boundary. It
differs from the line-search Newton-CG method described in
\cite[Section~7.1]{JNocedal_SJWright_2006} in that it makes use of
negative curvature directions when they are encountered, rather than
discarding them in favor of a steepest-descent direction.
Algorithm~\ref{alg:dncg} improves over both approaches in having a
global complexity theory for convergence to both approximate
first-order points, and points satisfying the approximate second-order
conditions \eqref{eq:2epson}.


\refer{In Section~\ref{sec:wcc}, we will analyze the global complexity
  properties of our algorithm. Local convergence could also be of
  interest, in particular, it is probably possible to prove rapid
  convergence of Algorithm~\ref{alg:dncg} once it reaches the
  neighborhood of a strict local minimum. 
  We believe that such results would be complicated and less enlightening 
  than the complexity guarantees, so we restrict
  our study to the latter.}

\section{Complexity Analysis}
\label{sec:wcc}

In this section, we present a global worst-case complexity analysis of
Algorithm~\ref{alg:dncg}. Elements of the analysis follow those in the
earlier paper \cite{CWRoyer_SJWright_2018}. The most technical part
appears in Section~\ref{subsec:wccccg} below, where we show that the
Capped CG procedure returns (deterministically) either an inexact
Newton step or a negative curvature direction, both of which can be
used as the basis of a successful backtracking line search. These
properties are used in Section~\ref{subsec:wcc1} to prove complexity
results for convergence to a point satisfying the approximate
first-order condition $\| \nabla f(x) \| \le
\epsg$. Section~\ref{subsec:wcc2} proves complexity results for
finding approximate second-order points \eqref{eq:2epson}, leveraging
properties of the minimum eigenvalue oracle, Procedure~\ref{alg:meo}.

\subsection{Properties of Capped CG}
\label{subsec:wccccg}

We now explore the properties of the directions $d$ that are output by
our Capped CG procedure, Algorithm~\ref{alg:ccg}. The main result
deals with the case in which Algorithm~\ref{alg:ccg} terminates due to
insufficiently rapid decrease in $\|r_j \|$, showing that the strategy
for identifying a direction of sufficient negative curvature for $H$
is effective.

\begin{theorem} \label{theo:cvCGwhilestronglycvx}
Suppose that the main loop of Algorithm~\ref{alg:ccg} terminates with
$j=\hIcg$, where
\[
\hIcg \in \{1,\dotsc,\min\{n,J(M,\eps,\zeta)\}\},
\]
(where $J(M,\eps,\zeta)$ is defined in Lemma~\ref{lemma:ccgits} and
\eqref{eq:HvcountcappedCG}) \refer{because the fourth termination test is
satisfied and the three earlier conditions do not hold, that is, 
$y_{\hIcg}^\top \bH y_{\hIcg} \ge \eps \|y_{\hIcg}\|^2$, 
$p_{\hIcg}^\top \bH p_{\hIcg} \ge \eps \|p_{\hIcg}\|^2$, and
\begin{equation} \label{eq:Kitsstronglycvx}
\|r_{\hIcg}\| > \max\{\hat{\zeta},\sqrt{T}\tau^{\hIcg /2}\}\|r_0\|.
\end{equation}
where $M$, $T$, $\hat{\zeta}$, and $\tau$ are the values returned by
Algorithm~\ref{alg:ccg}. Then
$y_{\hIcg+1}$ is computed by Algorithm~\ref{alg:ccg}, and we have}
\begin{equation} \label{eq:yKp1negcurv}
\frac{(y_{\hIcg +1}-y_i)^\top \bH (y_{\hIcg+1}-y_i)}
{\|y_{\hIcg+1}-y_i\|^2} < \eps, \quad \mbox{for some $i \in \{0,\dotsc,\hIcg-1\}$.}
\end{equation}
\end{theorem}

The proof of Theorem~\ref{theo:cvCGwhilestronglycvx} is quite
technical, and can be found in Appendix~\ref{app:ccg}. It relies on an
argument previously used to analyze a strategy based on accelerated
gradient~\cite[Appendix A.1]{YCarmon_JCDuchi_OHinder_ASidford_2017b},
itself inspired by a result of Bubeck~\cite{SBubeck_2015}, but it
needs some additional steps that relate specifically to CG.  The part
of our proof that corresponds
to~\cite[Appendix~A.1]{YCarmon_JCDuchi_OHinder_ASidford_2017b} is
simplified in some respects, thanks to the use of CG and the fact that
a quadratic \refer{(rather than a nonlinear) function is being
  minimized in the subproblem.}

Having shown that Algorithm~\ref{alg:ccg} is well-defined, we
summarize the properties of its outputs. 

\begin{lemma} \label{lemma:ccgsteps}
Let Assumptions~\ref{assum:compactlevelset} and~\ref{assum:fC22} hold,
and suppose that Algorithm~\ref{alg:ccg} is invoked at an iterate
$x_k$ of Algorithm~\ref{alg:dncg} (so that $\|\nabla f(x_k)\| >
\epsg>0$).
Let $d_k$ be the vector obtained in Algorithm~\ref{alg:dncg} from the
output $d$ of Algorithm~\ref{alg:ccg}. Then, one of the two following
statements holds:
\begin{enumerate}
\item {\em d\_type=SOL}, and the direction $d_k$ satisfies 
\begin{subequations} \label{eq:ccgstepSOL} 
\begin{align}
\label{eq:ccgstepSOLcurv}
d_k^\top (\nabla^2 f(x_k)+2\epsH I) d_k & \ge \epsH \|d_k\|^2, \\
\label{eq:ccgstepSOLnorm}
\|d_k\| & \le \refer{1.1 \epsH^{-1} \|\nabla f(x_k)\|,}\\
\label{eq:ccgstepSOLres}
\|\hat{r}_k \|  & \le  \refer{\frac{1}{2}\epsH\zeta \| d_k \|,}
\end{align}
\end{subequations}
where
\begin{equation} \label{eq:def.rhat}
\hat{r}_k := (\nabla^2 f(x_k)+2\epsH I) d_k+ \nabla f(x_k);
\end{equation}
\item {\em d\_type=NC}, and the direction $d_k$ satisfies 
  $d_k^\top \nabla f(x_k) \le 0$ as well as
  \refer{
  \begin{equation} \label{eq:ccgstepNC}
    \frac{d_k^\top \nabla^2 f(x_k) d_k}{\|d_k\|^2} = -\|d_k \| \le -\epsH.
  \end{equation}
  }
\end{enumerate}
\end{lemma}
\begin{proof}
\refer{For simplicity of notation, we use $H=\nabla^2 f(x_k)$ and
  $g=\nabla f(x_k)$ in the proof.}  Suppose first that d\_type=SOL. In
that case, we have from the termination conditions in
Algorithm~\ref{alg:ccg} and \eqref{eq:def.rhat} that
\begin{subequations} \label{eq:propSOLstep}
\begin{align} 
\label{eq:propSOLstepCURV}
d_k^\top (H+2 \epsH I)d_k & \ge \epsH \|d_k\|^2, \\
\label{eq:propSOLstepRES}
\|\hat{r}_k \| &  \le \hat{\zeta} \|g\|,
\end{align}
\end{subequations}
\refer{where $\hat\zeta$ was returned by the algorithm.}
We immediately recognize~\eqref{eq:ccgstepSOLcurv}
in~\eqref{eq:propSOLstepCURV}.  We now
prove~\eqref{eq:ccgstepSOLnorm}. Observe first 
that \eqref{eq:propSOLstepCURV} yields
\[
\epsH\|d_k\|^2 \le d_k^\top (H+2 \epsH I) d_k \le \|d_k\| \|(H+2\epsH
I) d_k\|,
\]
so from \eqref{eq:def.rhat} we have
\begin{equation*}
\|d_k\| \le \epsH^{-1} \|(H+2\epsH I)d_k\| = \epsH^{-1} \|-g+\hat{r}_k\| 
= \epsH^{-1} \sqrt{\|g\|^2 + \| \hat{r}_k \|^2} \le 
\epsH^{-1} \sqrt{1+\hat{\zeta}^2}  \|g\|,
\end{equation*}
where we used~\eqref{eq:propSOLstepRES} to obtain the final bound,
together with the equality $\|-g+\hat{r}_k\|^2=\|g\|^2+\| \hat{r}_k
\|^2$, which follows from $g^\top \hat{r}_k = r_0^\top \hat{r}_k =
0$, by orthogonality of the residuals in CG (see
Lemma~\ref{lemma:propertiesCGoneiter}, Property 2). Since
\refer{$\hat{\zeta} \le \zeta/(3 \kappa) \le 1/6$ by construction, we
  have $\|d_k \| \le \sqrt{37/36} \epsH^{-1} \|g\| \le 1.1 \epsH^{-1}
  \|g\|$, proving \eqref{eq:ccgstepSOLnorm}.}

The bound \eqref{eq:ccgstepSOLres} follows
from~\eqref{eq:propSOLstepRES} and the logic below:
\begin{align*}
	\| \hat{r}_k \| &\le \hat{\zeta} \|g\| \le 
	\hat{\zeta}\left(\|(H+2 \epsH I) d_k \| + 
	\| \hat{r}_k \|\right)  \le \hat\zeta \left( (M+2\epsH) \|d_k \| +  \| \hat{r}_k \| \right) \\
        \Rightarrow \;\; \| \hat{r}_k \| &\le 
	\frac{\hat{\zeta}}{1-\hat{\zeta}}(M+2 \epsH) \|d_k\|,
\end{align*}
\refer{where $M$ is the value returned by the algorithm.}
We finally use \refer{$\hat\zeta < 1/6$}
to arrive at
\[
\frac{\hat\zeta}{1-\hat{\zeta}} (M+2\epsH) \le \frac65 \hat{\zeta}
(M+2\epsH) = \frac65 \frac{\zeta \epsH}{3} < \frac12 \zeta \epsH,
\]
yielding  \eqref{eq:ccgstepSOLres}.

In the case of d\_type=NC, we recall that Algorithm~\ref{alg:dncg}
defines
\begin{equation} \label{eq:scaleccgNC}
d_k = -\mathrm{sgn}(d^\top g) \frac{|d^\top H
  d|}{\|d\|^2} \frac{d}{\|d\|}
\end{equation}
where $d$ denotes the direction obtained by
Algorithm~\ref{alg:ccg}. It follows immediately that $d_k^\top g \le
0$. Since $d_k$ and $d$ are collinear, we also have that
\refer{
\[
\frac{d_k^\top (H+2\epsH I)d_k}{\|d_k\|^2} = \frac{d^\top (H+2\epsH
  I)d}{\|d\|^2} \le \epsH \; \Rightarrow \; \frac{d_k^\top H
  d_k}{\|d_k\|^2}
\le -\epsH.
\]
By using this bound together with \eqref{eq:scaleccgNC}, we obtain
\[
\|d_k\| = \frac{|d^\top H d|}{\|d\|^2} = \frac{|d_k^\top H d_k|}{\|d_k\|^2} =
-\frac{d_k^\top H d_k}{\|d_k\|^2} \ge \epsH,
\]
proving \eqref{eq:ccgstepNC}.
}
\end{proof}

\subsection{First-Order Complexity Analysis}
\label{subsec:wcc1}

We now find a bound on the number of iterations and the amount of
computation required to identify an iterate $x_k$ for which $\| \nabla
f(x_k) \| \le \epsg$. We consider in turn the two types of steps
(approximate damped Newton and negative curvature), finding a lower
bound on the descent in $f$ achieved on the current iteration in each
case. We then prove an upper bound on the number of iterations
required to satisfy these approximate first-order conditions
(Theorem~\ref{theo:wcc1}) and an upper bound on the number of gradient
evaluations and Hessian-vector multiplications required
(Theorem~\ref{theo:wccits}).

We start with a lemma concerning the approximate damped Newton steps.
\begin{lemma} \label{lemma:LSccgSOL}
Suppose that Assumptions~\ref{assum:compactlevelset}
and~\ref{assum:fC22} hold. Suppose that at iteration $k$ of
Algorithm~\ref{alg:dncg}, we have $\|\nabla f(x_k)\| > \epsg$, so
that Algorithm~\ref{alg:ccg} is called. When Algorithm~\ref{alg:ccg}
outputs a direction $d_k$ with d\_type=SOL, then the backtracking line
search requires at most $j_k \le \jsol+1$ iterations, where
\begin{equation} \label{eq:ccglsitsSOL}
\jsol \; = \; \left[ \frac{1}{2}\log_{\theta}\left(
  \frac{3(1-\zeta)}{L_H+\eta} \frac{\epsH^2}{\refer{1.1} U_g} \right)
  \right]_+,
\end{equation}
and the resulting step $x_{k+1} = x_k + \alpha_k d_k$ satisfies
\begin{equation} \label{eq:ccglsdecreaseSOL}
f(x_k) - f(x_{k+1}) \; \ge \; 
\csol\min\left(\|\nabla f(x_{k+1})\|^3 \epsH^{-3},\epsH^3\right),
\end{equation}
where 
\begin{equation*}
\csol = \frac{\eta}{6}\min\left\{ \left[ \frac{4}{\sqrt{(4+\zeta)^2+8 L_H}+4+\zeta}
  \right]^3, \left[ \frac{3\theta^2 (1-\zeta) }{L_H+\eta}
  \right]^3\right\}.
\end{equation*}
\end{lemma}
\begin{proof}
The proof tracks closely that
of~\cite[Lemma~13]{CWRoyer_SJWright_2018}. The only significant
difference is that equation (65) of~\cite{CWRoyer_SJWright_2018},
which is instrumental to the proof and requires a probabilistic
assumption on $\lambda_{\min}(\nabla^2 f(x_k))$, is now ensured
\emph{deterministically} by~\eqref{eq:ccgstepSOLnorm} from
Lemma~\ref{lemma:ccgsteps}. As a result, both the proof and the result
are deterministic.
\end{proof}

When $\| \nabla f(x_{k+1}) \| \le \epsg$, the estimate
\eqref{eq:ccglsdecreaseSOL} may not guarantee a ``significant''
decrease in $f$ at this iteration.  However, in this case, the
approximate first-order condition $\| \nabla f(x) \| \le \epsg$ holds
at the {\em next} iteration, so that Algorithm~\ref{alg:dncg} will
invoke Procedure~\ref{alg:meo} at iteration $k+1$, leading either to
termination with satisfaction of the conditions \eqref{eq:2epson} or
to a step that reduces $f$ by a multiple of $\epsH^3$, as we show in
Theorem~\ref{theo:wccits} below.

We now address the case in which Algorithm~\ref{alg:ccg} returns a
negative curvature direction to Algorithm~\ref{alg:dncg} at iteration
$k$. The backtracking line search guarantees that a sufficient
decrease will be achieved at such an iteration. Although the Lipschitz
constant $L_H$ appears in our result, our algorithm (in contrast to
\cite{YCarmon_JCDuchi_OHinder_ASidford_2017b}) does not require
this constant to be known or estimated.

\begin{lemma} \label{lemma:LSccgNC}
  Suppose that Assumptions~\ref{assum:compactlevelset}
  and~\ref{assum:fC22} hold.  Suppose that at iteration $k$ of
  Algorithm~\ref{alg:dncg}, we have $\|\nabla f(x_k)\| > \epsg$, so
  that Algorithm~\ref{alg:ccg} is called. When Algorithm~\ref{alg:ccg}
  outputs d\_type=NC, the direction $d_k$ (computed from $d$ in
  Algorithm~\ref{alg:dncg}) has the following properties: The
  backtracking line search terminates with step length $\alpha_k =
  \theta^{j_k}$ with $j_k \le \jnc +1$, where
\begin{equation} \label{eq:eiglsits}	
\jnc := \left[ \log_{\theta}\left( \frac{3}{L_H+\eta} \right)
  \right]_+,
\end{equation}
and the resulting step $x_{k+1} = x_k + \alpha_k d_k$ satisfies
\begin{equation} \label{eq:decreaseNC}
f(x_k) - f(x_k+\alpha_k\,d_k) \; \geq \; \cnc \epsH^3,
\end{equation}	
with
\begin{equation*}
\cnc := \frac{\eta}{6}
\min\left\{1,\frac{27\theta^3}{(L_H+\eta)^3}\right\}.
\end{equation*}
\end{lemma}
\begin{proof}
  \refer{
By Lemma~\ref{lemma:ccgsteps}, we have from  \eqref{eq:ccgstepNC}
that
\begin{equation} \label{eq:boundnegcurvNCsteps}
		d_k^\top \nabla^2 f(x_k) d_k = -\|d_k\|^3 \le -\epsH \|d_k\|^2
\end{equation}
}
The result can thus be obtained exactly as in
\cite[Lemma~1]{CWRoyer_SJWright_2018}.
\end{proof}

\refer{We are ready to state our main result for first-order
  complexity.}

\begin{theorem} \label{theo:wcc1}
Let Assumptions~\ref{assum:compactlevelset} and~\ref{assum:fC22}
hold. Then, defining
\[
\bar{K}_1 := \left\lceil \frac{f(x_0)-\flow}{\min\{\csol,\cnc\}} 
\max\left\{\epsg^{-3}\epsH^3,\epsH^{-3}\right\} \right\rceil,
\]
some iterate $x_k$, $k=0,1,\dotsc,\bar{K}_1+1$ generated by
Algorithm~\ref{alg:dncg} will satisfy
\begin{equation} \label{eq:1epson}
 \|\nabla f(x_k)\| \le \epsg.
\end{equation} 
\end{theorem}
\begin{proof}
Suppose for contradiction that $\| \nabla f(x_k) \| > \epsg$ for all
$k=0,1,\dotsc,\bar{K}_1+1$, so that
\begin{equation} \label{eq:jh98}
  \| \nabla f(x_{l+1}) \| > \epsg, \quad l=0,1,\dotsc,\bar{K}_1.
\end{equation}
Algorithm~\ref{alg:ccg} will be invoked at each of the first
$\bar{K}_1+1$ iterates of Algorithm~\ref{alg:dncg}. For each iteration
$l=0,1,\dotsc,\bar{K}_1$ for which Algorithm~\ref{alg:ccg} returns
d\_type=SOL, we have from Lemma~\ref{lemma:LSccgSOL} and
\eqref{eq:jh98} that
\begin{equation} \label{eq:wcc1decSOL}
f(x_l) - f(x_{l+1}) \ge \csol \min\left\{\|\nabla
f(x_{l+1})\|^3\epsH^{-3}, \epsH^3\right\} \ge \csol
\min\left\{\epsg^3\epsH^{-3}, \epsH^3\right\}.
\end{equation}	
For each iteration $l=0,1,\dotsc,\bar{K}_1$ for which
Algorithm~\ref{alg:ccg} returns d\_type=NC, we have by
Lemma~\ref{lemma:LSccgNC} that
\begin{equation} \label{eq:wcc1decNC}
f(x_l) - f(x_{l+1}) \ge \cnc \epsH^3.
\end{equation}
By combining these results, we obtain
\begin{align*}
f(x_0) - f(x_{\bar{K}_1+1}) & \ge \sum_{l=0}^{\bar{K}_1} (f(x_l)-f(x_{l+1})) \\
& \ge \sum_{l=0}^{\bar{K}_1} \min \{\csol,\cnc\}\min\left\{\epsg^3\epsH^{-3},
  \epsH^3\right\} \\
  & = (\bar{K}_1+1)  \min \{\csol,\cnc\}\min\left\{\epsg^3\epsH^{-3},
  \epsH^3\right\} \\
  & >  f(x_0)-\flow.
\end{align*}
where we used the definition of $\bar{K}_1$ for the final
inequality. This inequality contradicts the definition of $\flow$ in
\eqref{eq:boundscompact}, so our claim is proved.
\end{proof}

If we choose $\epsH$ in the range $[\epsg^{1/3},\epsg^{2/3}]$, this bound
improves over the classical $\mathcal{O}(\epsg^{-2})$ rate of
gradient-based methods. 
The choice $\epsH=\epsg^{1/2}$ yields the rate
$\mathcal{O}(\epsg^{-3/2})$, which is known to be optimal among
second-order methods~\cite{CCartis_NIMGould_PhLToint_2011c}.

Recalling that the workload of Algorithm~\ref{alg:ccg} in terms of
Hessian-vector products depends on the index $J$ defined by
\eqref{eq:HvcountcappedCG}, we obtain the following
corollary. \refer{(Note the mild assumption on the quantities of $M$
  used at each instance of Algorithm~\ref{alg:ccg}, which is satisfied
  provided that this algorithm is always invoked with an initial
  estimate of $M$ in the range $[0,U_H]$.)}
%
\begin{corollary} \label{coro:wcc1Hv}
  Suppose that the assumptions of Theorem~\ref{theo:wcc1} are
  satisfied, and let $\bar{K}_1$ be as defined in that theorem and
  $J(M,\epsH,\zeta)$ be as defined in
  \eqref{eq:HvcountcappedCG}. \refer{Suppose that the values of $M$
    used or calculated at each instance of Algorithm~\ref{alg:ccg}
    satisfy $M \le U_H$.}
  Then the number of Hessian-vector
  products and/or gradient evaluations required by
  Algorithm~\ref{alg:dncg} to output an iterate
  satisfying~\eqref{eq:1epson} is at most
  \[
  \left(2\min \left\{n,J(U_H,\epsH,\zeta) \right\}+2 \right)(\bar{K}_1+1). 
  \]
  \refer{For $n$ sufficiently large, this bound is $ \tcO \left(
    \max\left\{\epsg^{-3} \epsH^{5/2},\epsH^{-7/2} \right\}\right)$,
    while if \\$J(U_H,\epsH,\zeta) \ge n$, the bound is $\tcO \left(
    n\,\max\left\{ \epsg^{-3} \epsH^{3},\epsH^{-3}\right\}\right)$.}
  \end{corollary}
\begin{proof}
  \refer{From Lemma~\ref{lemma:ccgits},} the number of Hessian-vector
  multiplications in the main loop of Algorithm~\ref{alg:ccg} is
  bounded by \refer{$\min \left\{ n, J(U_H,\epsH,\zeta)+1
    \right\}$}. An additional \refer{$\min \left\{n, J(U_H,\epsH,\zeta)
    \right\}$} Hessian-vector products may be needed to return a
  direction satisfying~\eqref{eq:weakcurvdir}, \refer{if
    Algorithm~\ref{alg:ccg} does not store its iterates $y_j$.} Each
  iteration also requires a single evaluation of the gradient $\nabla
  f$, giving a bound of \refer{$(2 \min
    \left\{n,J(U_H,\epsH,\zeta)\right\}+2)$} on the workload per
  iteration of Algorithm~\ref{alg:dncg}. We multiply this quantity by
  the iteration bound from Theorem~\ref{theo:wcc1} to obtain the
  result.
\end{proof}

By setting $\epsH=\epsg^{1/2}$, we obtain from this corollary a
computational bound of $\tcO(\epsg^{-7/4})$ \refer{(for $n$
  sufficiently large)}, which matches the deterministic first-order
guarantee obtained in \cite{YCarmon_JCDuchi_OHinder_ASidford_2017b},
and also improves over the $\mathcal{O}(\eps_g^{-2})$ computational
complexity of gradient-based methods.

\subsection{Second-Order Complexity Results}
\label{subsec:wcc2}

We now find bounds on iteration and computational complexity of
finding a point that satisfies \eqref{eq:2epson}. In this section, as
well as using results from Sections~\ref{subsec:wccccg}
and~\ref{subsec:wcc1}, we also need to use the properties of the
minimum eigenvalue oracle, Procedure~\ref{alg:meo}. To this end, we
make the following generic assumption.

\begin{assumption} \label{assum:wccmeo}
For every iteration $k$ at which Algorithm~\ref{alg:dncg} calls
Procedure~\ref{alg:meo}, and for a specified failure probability
$\delta$ with $0 \le \delta \ll 1$, Procedure~\ref{alg:meo} either
certifies that $\nabla^2 f(x_k) \succeq -\epsH I$ or finds a vector
of curvature smaller than $-{\epsH}/{2}$ in at most
\begin{equation} \label{eq:wccmeo}
		N_{\mathrm{meo}}:=\min\left\{n,
		1+\left\lceil\Cmeo\epsH^{-1/2}\right\rceil\right\}
\end{equation}
Hessian-vector products, \refer{with probability $1-\delta$,} where
$\Cmeo$ depends at most logarithmically on $\delta$ and $\epsH$.
\end{assumption}

%
%
Assumption~\ref{assum:wccmeo} encompasses the
strategies we mentioned in
Section~\ref{subsec:algomeo}. \refer{Assuming the bound $U_H$ on
  $\|H\|$ is available,} for both the Lanczos method with a random
starting vector and the conjugate gradient algorithm with a random
right-hand side, \eqref{eq:wccmeo} holds with
$\Cmeo=\ln(2.75n/\delta^2)\sqrt{U_H}/2$.  \refer{When a bound on
  $\|H\|$ is not available in advance, it can be estimated efficiently
  with minimal effect on the overall complexity of the method, as
  shown in Appendix~\ref{subapp:lacgwithoutM}.}

The next lemma guarantees termination of the backtracking line search
for a negative curvature direction, regardless of whether it is
produced by Algorithm~\ref{alg:ccg} or Procedure~\ref{alg:meo}. As in
Lemma~\ref{lemma:LSccgSOL}, the result is deterministic.
\begin{lemma} \label{lemma:LSgeneralNC}
Suppose that Assumptions~\ref{assum:compactlevelset}
and~\ref{assum:fC22} hold.  Suppose that at iteration $k$ of
Algorithm~\ref{alg:dncg}, the search direction $d_k$ is of negative
curvature type, obtained either directly from Procedure~\ref{alg:meo}
or as the output of Algorithm~\ref{alg:ccg} and d\_type=NC. Then the
backtracking line search terminates with step length $\alpha_k =
\theta^{j_k}$ with $j_k \le \jnc +1$, where $\jnc$ is defined as in
Lemma~\ref{lemma:LSccgNC}, and the decrease in the function value
resulting from the chosen step length satisfies
\begin{equation} \label{eq:LSgeneralNCdecrease}
f(x_k) - f(x_k+\alpha_k\,d_k) \; \geq \; \frac{\cnc}{8} \epsH^3,
\end{equation}	
with $\cnc$ is defined in Lemma~\ref{lemma:LSccgNC}.
\end{lemma}
\begin{proof}
  Lemma~\ref{lemma:LSccgNC} shows that the claim holds (with a factor
  of $8$ to spare) when the direction of negative curvature is
  obtained from Algorithm~\ref{alg:ccg}.  When the direction is
  obtained from Procedure~\ref{alg:meo}, we have \refer{by the scaling
    of $d_k$ applied in Algorithm~\ref{alg:dncg} that}
\begin{equation} \label{eq:boundnegcurvNCproba}
d_k^\top \nabla^2 f(x_k) d_k = -\|d_k\|^3 \le -\frac12 \epsH \|d_k
\|^2 < 0,
\end{equation}
from which it follows that $\|d_k \| \ge \tfrac12 \epsH$. The
result can now be obtained by following the proof of
Lemma~\ref{lemma:LSccgNC}, with $\tfrac12 \epsH$ replacing $\epsH$.
\end{proof}

\refer{We are now ready to state our iteration complexity result for
Algorithm~\ref{alg:dncg}.}
\begin{theorem} \label{theo:wccits}
Suppose that Assumptions~\ref{assum:compactlevelset},
\ref{assum:fC22}, and ~\ref{assum:wccmeo} hold, and define
\begin{equation} \label{eq:wccits}
\bar{K}_2 :=\left\lceil \frac{3(f(x_0)-\flow)}{\min\{\csol,\cnc/8\}}\max\{\epsg^{-3}
\epsH^3, \epsH^{-3}\} \right\rceil + 2,
\end{equation}
 where constants $\csol$ and $\cnc$
are defined in Lemmas~\ref{lemma:LSccgSOL} and~\ref{lemma:LSccgNC},
respectively.  Then with probability at least
\refer{$(1-\delta)^{\bar{K}_2}$}, Algorithm~\ref{alg:dncg} terminates
at a point satisfying \eqref{eq:2epson} in at most $\bar{K}_2$
iterations. (With probability at most
\refer{$1-(1-\delta)^{\bar{K}_2}$}, it terminates incorrectly within
$\bar{K}_2$ iterations at a point for which $\|\nabla f(x_k) \| \le
\epsg$ but $\lambda_{\min}(\nabla^2 f(x)) <
-\epsH$.)
\end{theorem}
\begin{proof}
Algorithm~\ref{alg:dncg} terminates incorrectly with probability
$\delta$ at any iteration at which Procedure~\ref{alg:meo} is called,
\refer{when Procedure~\ref{alg:meo} certifies erroneously that
  $\lambda_{\min}(\nabla^2 f(x)) \ge -\epsH$.}
\refer{Since an erroneous certificate can only lead to termination, an
  erroneous certificate at iteration $k$ means that
  Procedure~\ref{alg:meo} did not produce an erroneous certificate at
  iterations $0$ to $k-1$. By a disjunction argument, we have that the
  overall probability of terminating with an erroneous certificate
  during the first $\bar{K}_2$ iterations is bounded by
  $1-(1-\delta)^{\bar{K}_2}$. Therefore, with probability at least
  $(1-\delta)^{\bar{K}_2} $, no incorrect termination occurs in the
  first $\bar{K}_2$ iterations.} 
  
Suppose now for contradiction that Algorithm~\ref{alg:dncg} runs for
$\bar{K}_2$ iterations without terminating. That is, for all 
$l=0,1,\dotsc,\bar{K}_2$, we have either $\| \nabla f(x_l) \| > \epsg$ 
or $\lambdamin(\nabla^2 f(x_l)) < -\epsH$. We perform the following 
partition of the set of iteration indices:
\begin{equation} \label{eq:K123}
  \cK_1 \cup \cK_2 \cup \cK_3 = \{0,1,\dotsc,\bar{K}_2-1\},
  \end{equation}
where $\cK_1$, $\cK_2$, and $\cK_3$ are defined as follows.
  	
\medskip
	
\textbf{Case 1:} $\cK_1 := \{ l =0,1,\dotsc,\bar{K}_2-1 \, : \,
\|\nabla f(x_l) \| \le \epsg \}$.  At each iteration $l \in \cK_1$,
Algorithm~\ref{alg:dncg} calls Procedure~\ref{alg:meo}, which does not
certify that $\lambdamin (\nabla^2 f(x_l)) \ge -\epsH$ (since the
algorithm continues to iterate) but rather returns a direction of
sufficient negative curvature. By Lemma~\ref{lemma:LSgeneralNC}, the
step along this direction leads to an improvement in $f$ that is
bounded as follows:
\begin{equation} \label{eq:wccitsK1decrease}
f(x_l)-f(x_{l+1})  \ge  \frac{\cnc}{8} \epsH^3.
\end{equation}		
	
\medskip	
	
\textbf{Case 2:} $\cK_2 := \{ l=0,1,\dotsc,\bar{K}_2-1 \, : \, \|
\nabla f(x_l) \| > \epsg \; \mbox{and} \; \| \nabla f(x_{l+1}) \| >
\epsg \}$.  Algorithm~\ref{alg:dncg} calls Algorithm~\ref{alg:ccg} at
each iteration $l \in \cK_2$, returning either an approximate damped
Newton or a negative curvature direction. By combining
Lemmas~\ref{lemma:LSccgSOL} and \ref{lemma:LSccgNC}, we obtain a 
decrease in $f$ satisfying
\begin{align}
  \nonumber
f(x_l) - f(x_{l+1})  & \ge  \min\{\csol,\cnc\}
\min\left\{\| \nabla f(x_{l+1}) \|^3\epsH^{-3},\epsH^3\right\} \\
  \label{eq:wccitsK2decrease}
& \ge \min\{\csol,\cnc/8\}
\min\left\{\epsg^3\epsH^{-3},\epsH^3\right\}.
\end{align}

\medskip
	
\textbf{Case 3:} $\cK_3 := \{ l=0,1,\dotsc,\bar{K}_2-1 \, : \, \|
\nabla f(x_l) \| > \epsg \ge \| \nabla f(x_{l+1}) \| \}$.  Because $\|
\nabla f(x_{l+1})\|$ may be small in this case, we can no longer bound
the decrease in $f$ by an expression such as
\eqref{eq:wccitsK2decrease}. We can however guarantee at least that
$f(x_l)-f(x_{l+1}) \ge 0$. Moreover, provided that $l<\bar{K}_2-1$, we
have from $\| \nabla f(x_{l+1}) \| \le \epsg$ that the next iterate
$l+1$ is in $\cK_1$. Thus, a significant decrease in $f$ will be
attained at the {\em next} iteration, and we have
\begin{equation} \label{eq:K13}
  | \cK_3| \le |\cK_1|+1.
\end{equation}

\medskip

We now consider the total decrease in $f$ over the span of $\bar{K}_2$
iterations, which is bounded by $f(x_0)-\flow$ as follows:
\begin{align} 
  f(x_0) - \flow &\ge \sum_{l=0}^{\bar{K}_2-1} (f(x_l)-f(x_{l+1})) \nonumber \\
  \label{eq:sumdecompwccits}
& \ge \sum_{l \in \cK_1} (f(x_l)-f(x_{l+1})) +  \sum_{l \in \cK_2} (f(x_l)-f(x_{l+1})) 
\end{align}
where both sums in the final expression are nonnegative. Using first
the bound \eqref{eq:wccitsK1decrease} for the sum over $\cK_1$, we
obtain
\begin{equation} \label{eq:wccitsboundK1}
f(x_0) - \flow \; \ge \; |\cK_1| \frac{\cnc}{8} \epsH^3 \; \Leftrightarrow \; 
|\cK_1| \le \frac{f(x_0) - \flow}{\cnc/8} \epsH^{-3}.
\end{equation}
Applying~\eqref{eq:wccitsK2decrease} to the sum over $\cK_2$ leads to
\begin{equation} \label{eq:wccitsboundK2}
|\cK_2| \le \frac{f(x_0) -
  \flow}{\min\{\csol,\cnc/8\}}\max\{\epsg^{-3} \epsH^3,\epsH^{-3}\}.
\end{equation}
Using these bounds together with \eqref{eq:K13}, we have
\begin{align*}
  \bar{K}_2  & = |\cK_1| + |\cK_2| + |\cK_3| \\
  & \le  2|\cK_1|
  + |\cK_2| + 1 \\
  & \le 3\max\{|\cK_1|,|\cK_2|\} +1  \\
  & \le
\frac{3(f(x_0) - \flow)}
     {\min\{\csol,\cnc/8\}}\max\{\epsg^{-3}\epsH^3,\epsH^{-3}\} +1 \\
     & \le \bar{K}_2-1,
                \end{align*}
                giving the required contradiction.
\end{proof}

\refer{We note that when $\delta < 1/\bar{K}_2$ in
  Theorem~\ref{theo:wccits}, a technical result shows that
  $(1-\delta)^{\bar{K}_2} \ge 1-\delta \bar{K}_2$. In this case, the
  qualifier ``with probability at least $(1-\delta)^{\bar{K}_2}$ in
  the theorem can be replaced by ``with probability at least $1-\delta
  \bar{K}_2$'' while remaining informative.}

\refer{Finally, we provide an operation complexity result: a bound on the
number of Hessian-vector products and gradient evaluations necessary
for Algorithm~\ref{alg:dncg} to find a point that
satisfies~\eqref{eq:2epson}.}
\begin{corollary} \label{coro:wcc2Hv}
  Suppose that assumptions of Theorem~\ref{theo:wccits} hold, and let
  $\bar{K}_2$ be defined as in \eqref{eq:wccits}.  \refer{Suppose that
    the values of $M$ used or calculated at each instance of
    Algorithm~\ref{alg:ccg} satisfy $M \le U_H$.} Then with
  probability at least $(1-\delta)^{\bar{K}_2}$,
  Algorithm~\ref{alg:dncg} terminates
  at a point satisfying \eqref{eq:2epson} after at most
  \[
  \left(\max \{ 2\min\{n,J(U_H,\epsH,\zeta)\}+2,N_{\mathrm{meo}} \} \right) 
  \bar{K}_2
  \]
  Hessian-vector products and/or gradient evaluations. (With
  probability at most $1-(1-\delta)^{\bar{K}_2}$, it terminates incorrectly
  with this complexity at a point for which $\|\nabla f(x_k) \| \le
  \epsg$ but $\lambda_{\min}(\nabla^2 f(x)) < -\epsH$.)
  
  \refer{For $n$ sufficiently large, and assuming that 
  $\delta < 1/\bar{K}_2$, 
    the bound is\\
    $\tcO\left( \max\left\{\epsg^{-3} \epsH^{5/2},\epsH^{-7/2}
    \right\}\right)$, with probability at most $1-\bar{K}_2 \delta$.} 
\end{corollary}
\begin{proof}
The proof follows by combining Theorem~\ref{theo:wccits} (which bounds
the number of iterations) with Lemma~\ref{lemma:ccgits} and
Assumption~\ref{assum:wccmeo} (which bound the workload per
iteration).
  \end{proof}

By setting $\epsH=\epsg^{1/2}$ and assuming that $n$ is sufficiently
large, we recover (with high probability) the familiar complexity
bound of order $\tcO (\epsg^{-7/4})$, matching the bound of
accelerated gradient-type methods such
as~\cite{NAgarwal_ZAllenZhu_BBullins_EHazan_TMa_2017,YCarmon_JCDuchi_OHinder_ASidford_2018a,CJin_PNetrapalli_MIJordan_2018}.

\section{Discussion} \label{sec:discussion}

We have presented a Newton-CG approach for smooth nonconvex
unconstrained minimization that is close to traditional variants of
this method, but incorporates additional checks and safeguards that
enable convergence to a point satisfying approximate second-order
conditions \eqref{eq:2epson} with guaranteed complexity.  This was
achieved by exploiting the properties of \refer{Lanczos-based methods}
in two ways. First, we used CG to compute Newton-type steps when
possible, while monitoring convexity during the CG iterations to
detect negative curvature directions when those exist. Second, by
exploiting the close relationship between the Lanczos and CG
algorithms, we show that both methods can be used to detect negative
curvature of a given symmetric matrix with high probability. Both
techniques are endowed with complexity guarantees, and can be combined
within a Newton-CG framework to match the best known bounds for
second-order algorithms on nonconvex
optimization~\cite{CCartis_NIMGould_PhLToint_2017d}.

Nonconvexity detection can be introduced into CG in ways other than
those used in Algorithm~\ref{alg:ccg}. For instance, we can drop the
\emph{implicit} cap on the number of CG iterations that is due to
monitoring of the condition $\| r_j \| > \sqrt{T} \tau^{j/2} \|r_0 \|$
and use of the negative curvature direction generation procedure
\eqref{eq:weakcurvdir} from Algorithm~\ref{alg:ccg}, and instead
impose an \emph{explicit} cap (smaller by a factor of approximately
$4$ than $J(M,\eps,\zeta)$) on the number of CG iterations. In this
version, if the explicit cap is reached without detection of a
direction of sufficient negative curvature for $\bH$, then
Procedure~\ref{alg:meo} is invoked to find one. This strategy comes
equipped with essentially the same high-probability complexity results
as Theorem~\ref{theo:wccits} and Corollary~\ref{coro:wcc2Hv}, but it
lacks the deterministic approximate-first-order complexity guarantee
of Theorem~\ref{theo:wcc1}. On the other hand, it is more elementary,
both in the specification of the Capped CG procedure and the analysis.

A common feature to the Capped CG procedures described in
Algorithm~\ref{alg:ccg} and in the above paragraph, which also emerges
in most Newton-type methods with good complexity
guarantees~\cite{CCartis_NIMGould_PhLToint_2017d}, is the need for
high accuracy in the step computation. That is, only a small residual
is allowed in the damped Newton system at the approximate solution.
Looser restrictions are typically used in practical algorithms, but
our tighter bounds appear to be necessary for the complexity
analysis. Further investigation of the differences between our
procedure in this paper and practical Newton-CG procedures is a
subject of ongoing research.

\section*{Acknowledgments}

We thank sincerely the associate editor and two referees, whose
comments led us to improve the presentation and to derive stronger
results.

\bibliographystyle{siam} 
\bibliography{refs-newtonlanczoscg}

\begin{thebibliography}{10}

\bibitem{NAgarwal_ZAllenZhu_BBullins_EHazan_TMa_2017}
{\sc N.~Agarwal, Z.~{Allen-Zhu}, B.~Bullins, E.~Hazan, and T.~Ma}, {\em Finding
  approximate local minima faster than gradient descent}, in Proceedings of the
  49th Annual ACM SIGACT Symposium on Theory of Computing (STOC 2017), PMLR,
  2017.

\bibitem{ZAllenZhu_YLi_2018}
{\sc Z.~{Allen-Zhu} and Y.~Li}, {\em {NEON2}: {F}inding local minima via
  first-order oracles}, in Proceedings of the 32nd Conference on Neural
  Information Processing Systems, 2018.

\bibitem{EGBirgin_JMMartinez_2017}
{\sc E.~G. {Birgin} and J.~M. {Mart\'{i}nez}}, {\em The use of quadratic
  regularization with a cubic descent condition for unconstrained
  optimization}, SIAM J. Optim., 27 (2017), pp.~1049--1074.

\bibitem{SBubeck_2015}
{\sc S.~Bubeck}, {\em Convex optimization: {A}lgorithms and complexity},
  Foundations and Trends$^{\mathrm{\textcopyright}}$ in Machine Learning, 8
  (2015), pp.~231--357.

\bibitem{YCarmon_JCDuchi_OHinder_ASidford_2017b}
{\sc Y.~Carmon, J.~C. {Duchi}, O.~Hinder, and A.~Sidford}, {\em ``{C}onvex
  until proven guilty": {D}imension-free acceleration of gradient descent on
  non-convex functions}, in Volume 70: International Conference on Machine
  Learning, 6-11 August 2017, International Convention Centre, Sydney,
  Australia, PMLR, 2017, pp.~654--663.

\bibitem{YCarmon_JCDuchi_OHinder_ASidford_2018a}
\leavevmode\vrule height 2pt depth -1.6pt width 23pt, {\em Accelerated methods
  for non-convex optimization}, SIAM J. Optim., 28 (2018), pp.~1751--1772.

\bibitem{CCartis_NIMGould_PhLToint_2010}
{\sc C.~Cartis, N.~I.~M. {Gould}, and P.~L. {Toint}}, {\em On the complexity of
  steepest descent, {N}ewton's and regularized {N}ewton's methods for nonconvex
  unconstrained optimization}, SIAM J. Optim., 20 (2010), pp.~2833--2852.

\bibitem{CCartis_NIMGould_PhLToint_2011a}
\leavevmode\vrule height 2pt depth -1.6pt width 23pt, {\em Adaptive cubic
  regularisation methods for unconstrained optimization. {P}art {I}:
  motivation, convergence and numerical results}, Math. Program., 127 (2011),
  pp.~245--295.

\bibitem{CCartis_NIMGould_PhLToint_2011c}
\leavevmode\vrule height 2pt depth -1.6pt width 23pt, {\em Optimal
  {N}ewton-type methods for nonconvex optimization}, Tech. Rep. naXys-17-2011,
  Dept of Mathematics, FUNDP, Namur (B), 2011.

\bibitem{CCartis_NIMGould_PLToint_2012}
\leavevmode\vrule height 2pt depth -1.6pt width 23pt, {\em Complexity bounds
  for second-order optimality in unconstrained optimization}, Journal of
  Complexity, 28 (2012), pp.~93--108.

\bibitem{CCartis_NIMGould_PhLToint_2017d}
\leavevmode\vrule height 2pt depth -1.6pt width 23pt, {\em Worst-case
  evaluation complexity and optimality of second-order methods for nonconvex
  smooth optimization}.
\newblock arXiv:1709.07180, 2017.

\bibitem{ARConn_NIMGould_PhLToint_2000}
{\sc A.~R. {Conn}, N.~I.~M. {Gould}, and P.~L. {Toint}}, {\em Trust-Region
  Methods}, MPS-SIAM Series on Optimization, Society for Industrial and Applied
  Mathematics, Philadelphia, 2000.

\bibitem{FECurtis_DPRobinson_MSamadi_2017}
{\sc F.~E. {Curtis}, D.~P. {Robinson}, and M.~Samadi}, {\em A trust region
  algorithm with a worst-case iteration complexity of
  $\mathcal{O}\left(\epsilon^{-3/2}\right)$ for nonconvex optimization}, Math.
  Program., 162 (2017), pp.~1--32.

\bibitem{FECurtis_DPRobinson_MSamadi_2018}
\leavevmode\vrule height 2pt depth -1.6pt width 23pt, {\em An inexact
  regularized {N}ewton framework with a worst-case iteration complexity of
  $\mathcal{O}(\epsilon^{-3/2})$ for nonconvex optimization}, IMA J. Numer.
  Anal.,  (2018 (to appear)).

\bibitem{RSDembo_TSteihaug_1983}
{\sc R.~S. Dembo and T.~Steihaug}, {\em Truncated-{N}ewton algorithms for
  large-scale unconstrained optimization}, Math. Program., 26 (1983),
  pp.~190--212.

\bibitem{GFasano_SLucidi_2009}
{\sc G.~Fasano and S.~Lucidi}, {\em A nonmonotone truncated {Newton-Krylov}
  method exploiting negative curvature directions, for large-scale
  unconstrained optimization}, Optim. Lett., 3 (2009), pp.~521--535.

\bibitem{NIMGould_SLucidi_MRoma_PhLToint_2000}
{\sc N.~I.~M. {Gould}, S.~Lucidi, M.~Roma, and P.~L. {Toint}}, {\em Exploiting
  negative curvature directions in linesearch methods for unconstrained
  optimization}, Optim. Methods Softw., 14 (2000), pp.~75--98.

\bibitem{AGriewank_AWalther_2008}
{\sc A.~Griewank and A.~Walther}, {\em Evaluating Derivatives: {P}rinciples and
  Techniques of Algorithmic Differentiation}, Frontiers in Applied Mathematics,
  SIAM, Philadelphia, PA, second~ed., 2008.

\bibitem{CJin_PNetrapalli_MIJordan_2018}
{\sc C.~Jin, P.~Netrapalli, and M.~I. {Jordan}}, {\em Accelerated gradient
  descent escapes saddle points faster than gradient descent}, in Proceedings
  of the 31st Conference On Learning Theory, PMLR, 2018, pp.~1042--1085.

\bibitem{SKarimi_SAVavasis_2016}
{\sc S.~Karimi and S.~A. {Vavasis}}, {\em A unified convergence bound for
  conjugate gradient and accelerated gradient}.
\newblock arXiv:1605.00320, 2016.

\bibitem{SKarimi_SAVavasis_2017}
\leavevmode\vrule height 2pt depth -1.6pt width 23pt, {\em A single potential
  governing convergence of conjugate gradient, accelerated gradient and
  geometric descent}.
\newblock arXiv:1712.09498, 2017.

\bibitem{JKuczynski_HWozniakowski_1992}
{\sc J.~Kuczy\'{n}ski and H.~Wo\'{z}niakowski}, {\em Estimating the largest
  eigenvalue by the power and {Lanczos} algorithms with a random start}, SIAM
  J. Matrix Anal. Appl., 13 (1992), pp.~1094--1122.

\bibitem{JMMartinez_MRaydan_2017}
{\sc J.~M. {Mart\'{i}nez} and M.~Raydan}, {\em Cubic-regularization counterpart
  of a variable-norm trust-region method for unconstrained minimization}, J.
  Global Optim., 68 (2017), pp.~367--385.

\bibitem{YuNesterov_BTPolyak_2006}
{\sc Y.~Nesterov and B.~T. {Polyak}}, {\em Cubic regularization of {Newton}
  method and its global performance}, Math. Program., 108 (2006), pp.~177--205.

\bibitem{JNocedal_SJWright_2006}
{\sc J.~Nocedal and S.~J. {Wright}}, {\em Numerical {O}ptimization}, Springer
  Series in Operations Research and Financial Engineering, Springer-Verlag, New
  York, second~ed., 2006.

\bibitem{CWRoyer_SJWright_2018}
{\sc C.~W. {Royer} and S.~J. {Wright}}, {\em Complexity analysis of
  second-order line-search algorithms for smooth nonconvex optimization}, SIAM
  J. Optim., 28 (2018), pp.~1448--1477.

\bibitem{TSteihaug_1983}
{\sc T.~Steihaug}, {\em The conjugate gradient method and trust regions in
  large scale optimization}, SIAM J. Numer. Anal., 20 (1983), pp.~626--637.

\bibitem{YXu_RJin_TYang_2018}
{\sc Y.~Xu, R.~Jin, and T.~Yang}, {\em First-order stochastic algorithms for
  escaping from saddle points in almost linear time}, in Proceedings of the
  32nd Conference on Neural Information Processing Systems, 2018.

\end{thebibliography}


\appendix

\section{Linear Conjugate Gradient: Relevant Properties} \label{app:cgresults}

In this appendix, we provide useful results for the classical CG
algorithm, that
also apply to the ``standard CG'' operations within
\refer{Algorithm~\ref{alg:ccg}}. To this end, and for the sake of
discussion in Appendix~\ref{app:lacg}, we sketch the standard CG
method in Algorithm~\ref{alg:cg}, reusing the notation of
Algorithm~\ref{alg:ccg}.

\begin{algorithm}[ht!]
\caption{Conjugate Gradient}
\label{alg:cg}
\begin{algorithmic}
  \STATE \emph{Inputs:} Symmetric matrix $\bH$, vector $g$;
  \STATE $r_0 \leftarrow g$, $p_0 \leftarrow -r_0$, $y_0 \leftarrow 0$, $j \leftarrow 0$;
  \WHILE{$p_j^\top\bH p_j>0$ and $r_j \ne 0$}
  \STATE $\alpha_j \leftarrow {r_j^\top r_j}/{p_j^\top \bH p_j}$;
\STATE $y_{j+1} \leftarrow y_j+\alpha_j p_j$;
\STATE $r_{j+1} \leftarrow r_j + \alpha_j \bH p_j$;
\STATE $\beta_{j+1} \leftarrow {(r_{j+1}^\top r_{j+1})}/{(r_j^\top r_j)}$;
\STATE $p_{j+1} \leftarrow -r_{j+1} + \beta_{j+1}p_j$;
\STATE $j \leftarrow j+1$;
\ENDWHILE
\end{algorithmic}
\end{algorithm}

Here and below, we refer often to the following quadratic function:
\begin{equation} \label{eq:quadraticCG}
	q(y) := \frac{1}{2}y^\top \bH y + g^\top y,
\end{equation}
where $\bH$ and $g$ are the matrix and vector parameters of
Algorithms~\ref{alg:ccg} or~\ref{alg:cg}. When $\bH$ is positive
definite, the minimizer of $q$ is identical to the unique solution of
$\bH y = -g$.  CG can be viewed either as an algorithm to solve $\bH y
= -g$ or as an algorithm to minimize $q$.

The next lemma details several properties of the conjugate gradient
method to be used in the upcoming proofs.
\begin{lemma} \label{lemma:propertiesCGoneiter}
Suppose that $j$ iterations of the CG loop are performed in
Algorithm~\ref{alg:ccg} or~\ref{alg:cg}. Then, we have
\begin{equation} \label{eq:suffcurvuptoj}
	  \frac{p_i^\top \bH p_i}{\|p_i\|^2} > 0 \quad \mbox{for
       all $i=0,1,\dotsc,j-1$}.
\end{equation}	
Moreover, the following properties hold.
\begin{enumerate}
\item $y_i \in \lspan\left\{p_0,\dotsc,p_{i-1}\right\}$,
  $i=1,2,\dotsc,j$.
\item $r_i \in \lspan\left\{p_0,\dotsc,p_i\right\}$ for all $i=1,2,\dotsc,j$, and 
\[
r_i^\top v = 0, \quad \mbox{for all $v \in
  \lspan\left\{p_0,\dotsc,p_{i-1}\right\}$ and all $i=1,2,\dotsc,j$}.
\]
(In particular, $r_i^\top r_l = 0$ if \refer{$0 \le l<i \le j$. If} $j=n$, 
then $r_n=0$.)
\item $\|r_i\| \le \|p_i\|$, $i=0,1,\dotsc,j$.
\item $r_i^\top p_i = -\|r_i\|^2$, $i=0,1,\dotsc,j$.
\item $p_i^\top \bH p_k=0$ for all $i,k =0,1,\dotsc,j$ with $k \neq i$.
\item $p_i = -\sum_{k=0}^i ({\|r_i\|^2}/{\|r_k\|^2}) r_k$, $i=0,1,\dotsc,j$.
\item
  \[
  q(y_{i+1}) = q(y_i) - \frac{\|r_i\|^4}{2 p_i^\top \bH p_i}, \quad i=0,1,\dotsc,j-1.
  \]
\item $r_i^\top \bH r_i \ge p_i^\top \bH p_i$, $i=0,1,\dotsc,j$.
\end{enumerate}		
\end{lemma}	
\begin{proof}
Since CG has not terminated prior to iteration $j$,
\eqref{eq:suffcurvuptoj} clearly holds.  All properties then follow
from the definition of the CG process, and most are proved in standard
texts (see, for example,
\cite[Chapter~5]{JNocedal_SJWright_2006}). Property 8 is less commonly
used, so we provide a proof here.

The case $i=0$ is immediate since $r_0=-p_0$ and there is
equality. When $i \ge 1$, we have:
\begin{equation*}
	p_i = -r_i + \frac{\|r_i\|^2}{\|r_{i-1}\|^2}p_{i-1} \; \Leftrightarrow \; 
	r_i = -p_i + \frac{\|r_i\|^2}{\|r_{i-1}\|^2}p_{i-1}.
\end{equation*}
(Note that if iteration $i$ is reached, we cannot have
$\|r_{i-1}\|=0$.)  It follows that
\begin{eqnarray*}
	r_i^\top \bH r_i &=   
	&p_i^\top \bH p_i -2\frac{\|r_i\|^2}{\|r_{i-1}\|^2} p_i^\top \bH p_{i-1} + 
	\frac{\|r_i\|^4}{\|r_{i-1}\|^4}p_{i-1}^\top \bH p_{i-1} \\
	&= &p_i^\top \bH p_i + \frac{\|r_i\|^4}{\|r_{i-1}\|^4}p_{i-1}^\top \bH p_{i-1},
\end{eqnarray*}
as $p_i^\top \bH p_{i-1} = 0$ by Property 5 above. Since iteration $i$
has been reached, $p_{i-1}$ is a direction of \refer{positive
  curvature},
and we obtain $r_i^\top \bH r_i \ge p_i^\top \bH p_i$, as required.
\end{proof}

We next address an important technical point about
Algorithm~\ref{alg:ccg}: the test \eqref{eq:weakcurvdir} to identify a
direction of negative curvature for $H$ after an insufficiently rapid
rate of reduction in the residual norm $\|r_j\|$ has been observed. As
written, the formula \eqref{eq:weakcurvdir} suggests both that
previous iterations $y_i$, $i=1,2,\dotsc,j-1$ must be stored (or
regenerated) and that additional matrix-vector multiplications
(specifically, $\bH (y_{j+1}-y_i)$, $i=0,1,\dotsc$) must be
performed. We show here that in fact \eqref{eq:weakcurvdir} can be
evaluated at essentially no cost, provided we store two extra scalars
at each iteration of CG: the quantities $\alpha_k$ and
$\|r_k\|^2$, for $k=0,1,\dotsc,j$.
\begin{lemma} \label{lemma:cgcurvaturescalars}
Suppose that Algorithm~\ref{alg:ccg} computes iterates up to iteration
$j+1$. Then, for any $i \in \{0,\dotsc,j\}$, we can compute
(\ref{eq:weakcurvdir}) as
\[
\frac{(y_{j+1}-y_i)^\top \bH (y_{j+1}-y_i)}{\|y_{j+1}-y_i\|^2} =
\frac{\sum_{k=i}^{j} \alpha_k \|r_k\|^2}
{\sum_{\ell=0}^j \left[\sum_{k=\max\{\ell,i\}}^j \alpha_k \|r_k\|^2 \right]^2/\|r_{\ell}\|^2}.
\]
\end{lemma}
\begin{proof}
By definition, $y_{j+1}-y_i = \sum_{k=i}^{j} \alpha_k p_k$. By
conjugacy of the $p_k$ vectors, we have
\begin{equation} \label{eq:numeratorcomp}
(y_{j+1}-y_i)^\top \bH (y_{j+1}-y_i)
= \sum_{k=i}^{j} \alpha_k^2 p_k^\top \bH p_k
= \sum_{k=i}^{j} \alpha_k \|r_k\|^2,
\end{equation}
where we used the definition of $\alpha_k$ to obtain the last
equality.  Now we turn our attention to the denominator. Using
Property 6 of Lemma~\ref{lemma:propertiesCGoneiter}, we have that
\[
y_{j+1}-y_i = \sum_{k=i}^{j} \alpha_k p_k =
\sum_{k=i}^{j} \alpha_k \left(-\sum_{\ell=0}^k \frac{\|r_k\|^2}{\|r_\ell\|^2} r_\ell \right),
\]
By rearranging the terms in the sum, we obtain
\begin{equation*}
y_{j+1}-y_i = -\sum_{k=i}^{j} \sum_{\ell=0}^k \alpha_k\|r_k\|^2 \frac{r_\ell}{\|r_{\ell}\|^2} 
= -\sum_{\ell=0}^j \left[\sum_{k=\max\{\ell,i\}}^j \alpha_k \|r_k\|^2 \right] \frac{r_\ell}{\|r_\ell\|^2}.
\end{equation*}
Using the fact that the residuals $\{ r_{\ell} \}_{\ell=0,1,\dotsc,j}$
form an orthogonal set (by Property 2 of
Lemma~\ref{lemma:propertiesCGoneiter}), we have that
\begin{equation*}
  \|y_{j+1}-y_i\|^2 = \sum_{\ell=0}^j \frac{1}{\|r_{\ell} \|^2}
  \left[\sum_{k=\max\{\ell,i\}}^j \alpha_k \|r_k\|^2 \right]^2.
\end{equation*}
Combining this with~\eqref{eq:numeratorcomp} gives the desired result.
\end{proof}

\section{Implementing Procedure~\ref{alg:meo} via Lanczos and Conjugate Gradient} \label{app:lacg}

In the first part of this appendix (Appendix~\ref{subapp:RL}) we outline
the randomized Lanczos approach and describe some salient convergence
properties. The second part (Appendix~\ref{subapp:lacgwithM}) analyzes
the CG method (Algorithm~\ref{alg:cg}) applied to a (possibly
nonconvex) quadratic function with a random linear term. We show that
the number of iterations required by CG to detect nonpositive
curvature in an indefinite matrix is the same as the number required
by Lanczos, when the two approaches are initialized in a consistent
way, thereby proving Theorem~\ref{theo:randcg}.  \refer{As a result,
  both techniques are implementations of Procedure~\ref{alg:meo} that
  satisfy Assumption~\ref{assum:wccmeo}, provided than an upper bound
  $M$ on $\|H\|$ is known. In the third part
  (Appendix~\ref{subapp:lacgwithoutM}), we deal with the case in which
  a bound on $\|H\|$ is not known a priori, and describe a version of
  the randomized Lanczos scheme which obtains an overestimate of this
  quantity (to high probability) during its first phase of
  execution. The complexity of this version differs by only a modest
  multiple from the complexity of the original method, and still
  satisfies Assumption~\ref{assum:wccmeo}.}


\subsection{\refer{Randomized Lanczos}}
\label{subapp:RL}

Consider first the Lanczos algorithm applied to a symmetric,
$n$-by-$n$ matrix $\bH$ and a
starting vector $b \in \R^n$ with $\| b\|=1$.  After $t+1$ iterations,
Lanczos \refer{constructs a basis of the $t$-th Krylov subspace
  defined by
\begin{equation} \label{eq:Kk}
\cK_t(b, \bH) = \lspan \{ b, \bH b, \dotsc, \bH^t b \}.
\end{equation}
The Lanczos method can compute estimates of the minimum and maximum
eigenvalues of $\bH$. For $t=0,1,\dotsc$, those values are given by
\begin{subequations}  \label{eq:eigvalsKk}
\begin{align}     \label{eq:eigvalsKk.min}
\ximin(\bH,b,t) &= \min_z \, z^\top \bH z  \quad \mbox{subject to $\|z\|_2=1$,
   $z \in \cK_t(b,\bH)$,} \\
\ximax(\bH,b,t) &= \max_z  \, z^\top \bH z  \quad \mbox{subject to $\|z\|_2=1$,
   $z \in \cK_t(b,\bH)$.} 
\end{align}
\end{subequations}
The Krylov subspaces satisfy a shift invariance property, that is, for
any $\hat{H}=a_1 I + a_2 H$ with $(a_1,a_2) \in \R^2$, we have that
\begin{equation} \label{eq:shiftinvariance}
	\cK_t(b, \hat{H}) = \cK_t(b,H) \quad \mbox{for $t=0,1,\dotsc$}
\end{equation}
}

\refer{
Properties of the randomized Lanczos procedure are explored in
\cite{JKuczynski_HWozniakowski_1992}. The following key result is a
direct consequence of Theorem~4.2(a) from the cited paper, along with
the shift invariance property mentioned above.
\begin{lemma} \label{lem:KW4.2}
  Let $\bH$ be an $n \times n$ symmetric matrix, let $b$ be chosen
  from a uniform distribution over the sphere $\|b\|=1$, and suppose
  that $\bar{\eps} \in [0,1)$ and $\delta \in (0,1)$ are
    given. Suppose that $\ximin (\bH,b,k)$ and $\ximax(\bH,b,k)$ are
    defined as in \eqref{eq:eigvalsKk}. Then after $k$ iterations of
    randomized Lanczos, the following convergence condition holds:
    \begin{equation} \label{eq:KW5}
      \lambdamax(\bH) - \ximax(\bH,b,k) \le \bar{\eps}(\lambdamax(\bH)-\lambdamin(\bH)) 
      \quad \mbox{with probability at least $1-\delta$,}
    \end{equation}
  provided $k$ satisfies
    \begin{equation} \label{eq:KW6}
     k=n \quad \mbox{or} \quad 1.648 \sqrt{n} \exp \left( -\sqrt{\bar{\eps}} (2k-1) \right) \le \delta.
      \end{equation}
  A sufficient condition for \eqref{eq:KW5} thus is
    \begin{equation} \label{eq:KW7}
      k \ge \min\left \{n,1+\left\lceil \frac{1}{4 \sqrt{\bar{\eps}}} \ln (2.75 n/\delta^2) \right\rceil
      \right\}.
    \end{equation}
    Similarly, we have that
    \begin{equation} \label{eq:KW8}
      \ximin(\bH,b,k)-\lambdamin(\bH) \le \bar{\eps}(\lambdamax(\bH)-\lambdamin(\bH)) 
      \quad \mbox{with probability at least $1-\delta$}
      \end{equation}
    for $k$ satisfying the same conditions \eqref{eq:KW6} or \eqref{eq:KW7}.
\end{lemma}
}

\subsection{\refer{Lanczos and Conjugate Gradient as Minimum Eigenvalue Oracles}}
\label{subapp:lacgwithM}

\refer{ Lemma~\ref{lemma:randlanczos} implies that using the Lanczos
  algorithm to generate the minimum eigenvalue of $\bH$ from
  \eqref{eq:eigvalsKk.min} represents an instance of
  Procedure~\ref{alg:meo} satisfying Assumption~\ref{assum:wccmeo}.
  The sequence of iterates $\{z_t\}$ given by $z_0=b$ and
\begin{equation} \label{eq:lanc.obj}
z_{t+1}\, \refer{\in}\, \arg\min_z \, \frac12 z^\top \bH z \quad \mbox{subject to $\|z\|_2=1$,
   $z \in \cK_t(b,\bH)$, for $t=0,1,\dotsc$}
\end{equation}
 eventually yields a direction of sufficient negative curvature, when
 such a direction exists.}

Consider now Algorithm~\ref{alg:cg} applied to $\bH$, and $g=-b$.  By
Property 2 of Lemma~\ref{lemma:propertiesCGoneiter}, we can
see that if Algorithm~\ref{alg:cg} does not terminate with $j \le t$,
for some given index $t$, then $y_{t+1}$, $r_{t+1}$, and $p_{t+1}$ are
computed, and we have
\begin{equation} \label{eq:CGspanKspace}
	\lspan\{p_0,\dotsc,p_i\} = \lspan\{r_0,\dotsc,r_i\} 
	= \cK_{i}(b,\bH), \quad \mbox{for $i=0,1,\dotsc,t$,}
\end{equation}
because $\{r_{\ell}\}_{\ell=0}^{i}$ is a set of $i+1$ orthogonal
vectors in $\cK_{i}(b,\bH)$. Thus $\{p_0,\dotsc,p_i\}$,
$\{r_0,\dotsc,r_i\}$, and $\{ b, \bH b, \dotsc, \bH^i b \}$ are all
bases for $\cK_{i}(b,\bH)$, $i=0,1,\dotsc,t$.
As long as they are computed, the iterates of
Algorithm~\ref{alg:cg} satisfy
\begin{equation} \label{eq:CGsol}
y_{t+1} := \arg\min_y \, \frac12 y^\top \bH y - b^\top y \quad \mbox{subject to    
  $y \in \cK_{t}(b,\bH)$, for $t=0,1,\dotsc$.}
\end{equation}
\refer{ The sequences defined by~\eqref{eq:lanc.obj} (for Lanczos)
  and~\eqref{eq:CGsol} (for CG) are related via the Krylov subspaces.}
We have the following result about the number of iterations required
by CG to detect non-positive-definiteness.
\begin{theorem} \label{th:lcg}
  Consider applying Algorithm~\ref{alg:cg} to the quadratic function
  \eqref{eq:quadraticCG}, with $g=-b$ for some $b$ with $\|b\| = 1$.
  Let $\JB$ be the smallest value of $t \ge 0$ such that
  $\cK_t(b,\bH)$ contains a direction of nonpositive curvature, so
  that $\JB$ is also the smallest index $t \ge 0$ such that $z_{t+1}^\top
  \bH z_{t+1} \le 0$, where $\{z_j \}$ are the Lanczos iterates from
  \eqref{eq:lanc.obj}. Then Algorithm~\ref{alg:cg} terminates with
  $j=\JB$, with $p_\JB^\top \bH p_\JB \le 0$.
\end{theorem}
\begin{proof}
  We consider first the case of $\JB=0$. Then $z_1 = b / \|b\|$ and
  $b^\top \bH b \le 0$, so since $p_0=-r_0=b$, we have $p_0^\top \bH p_0 \le
  0$, so the result holds in this case. We assume that $\JB \ge 1$ for
  the remainder of the proof.
  
  Suppose first that Algorithm~\ref{alg:cg} terminates with $j=t$, for
  some $t$ satisfying $1 \le t \le \JB$, because of a zero residual
  --- $r_t=0$ --- without having encountered nonpositive curvature.
  In that case, we can show that 
  $\bH^t b \in \lspan\{b,\dotsc,\bH^{t-1}b\}$.

     We can invoke~\eqref{eq:CGspanKspace} with $t$ replaced by $t-1$ 
     since Algorithm~\ref{alg:cg} has not terminated at iteration $t-1$. 
     By the recursive definition of $r_{t-1}$ 
     within Algorithm~\ref{alg:cg}, 
     there exist coefficients $\tau_i$ and
     $\sigma_i$ such that
     \[
     r_{t-1} = -b + \sum_{i=1}^{t-1} \tau_i \bH^i b, \quad
     p_{t-1} = \sum_{i=0}^{t-1} \sigma_i \bH^i b.
     \]
     Since $r_t=0$, we have again from Algorithm~\ref{alg:cg} that
     \begin{equation} \label{eq:uh7}
     0 = r_{t-1} + \alpha_{t-1} \bH p_{t-1} = -b + \sum_{i=1}^{t-1} (\tau_i + \alpha_{t-1} \sigma_{i-1}) \bH^i b +  \alpha_{t-1} \sigma_{t-1} \bH^t b.
     \end{equation}
     The coefficient $\alpha_{t-1} \sigma_{t-1}$ must be nonzero,
     because otherwise this expression would represent a nontrivial
     linear combination of the basis elements $\{ b, \bH b, \dotsc,
     \bH^{t-1} b \}$ of $\cK_{t-1}(b,\bH)$ that is equal to zero. It
     follows from this observation and \eqref{eq:uh7} that $\bH^t b \in \lspan \{b, \bH b,
     \dotsc, \bH^{t-1} b \} = \cK_{t-1}(b,\bH)$, as required. 
   

  Consequently,
  \[
	\cK_t(b,\bH) = \lspan\{b,\bH b,\dotsc,\bH^t b\} =
	\lspan\{b,\dotsc,\bH^{t-1}b\} = \cK_{t-1}(b,\bH).
  \] 
  By using a recursive argument on the definition of $\cK_i(b,\bH)$
  for $i=t,\dotsc,\JB$, we arrive at $\cK_{t-1}(b,\bH) =
  \cK_{\JB}(b,\bH)$. Thus there is a value of $t$ smaller than $\JB$
  such that $\cK_{t}(b,\bH)$ contains a direction of nonpositive
  curvature, contradicting our definition of $\JB$. Thus we cannot
  have termination of Algorithm~\ref{alg:cg} with $j \le \JB$ unless
  $p_j^\top \bH p_j \le 0$.

  Suppose next that CG terminates with $j=t$ for some $t >\JB$. It
  follows that $p_j^\top \bH p_j > 0$ for all $j=0,1,\dotsc,\JB$.  By
  definition of $\JB$,  there is a nonzero vector $z \in
  \cK_{\JB}(b,\bH)$ such that $z^\top \bH z \le 0$. On the other hand,
  we have $\cK_{\JB}(b,\bH) = \lspan\{p_0,p_1,\dotsc,p_{\JB} \}$
  by~\eqref{eq:CGspanKspace}, thus we can write $z = \sum_{j=0}^{\JB}
  \gamma_j p_j$, for some scalars $\gamma_j$, $j=0,1,\dotsc,\JB$. By
  Property 5 of Lemma~\ref{lemma:propertiesCGoneiter}, we then have
\[
0 \ge z^\top \bH z = \sum_{j=0}^{\JB} \gamma_j^2 p_j^\top \bH p_j.
\]
Since $p_{j}^\top \bH p_{j}>0$ for every $j=0,1,\dotsc,\JB$, and not
all $\gamma_j$ can be zero (because $z \ne 0$), the final summation is
strictly positive, a contradiction.

Suppose now that CG terminates at some $j<\JB$. Then $p_j^\top \bH p_j
\le 0$, and since $p_j \in \cK_j(b,\bH)$, it follows that
$\cK_j(b,\bH)$ contains a direction of nonpositive curvature,
contradicting the definition of $\JB$.

We conclude that Algorithm~\ref{alg:cg} must terminate with $j=\JB$
and $p_\JB^\top \bH p_\JB \le 0$, as claimed.
\end{proof}

Theorem~\ref{th:lcg} is a generic result that does not require $b$ to
be chosen randomly. It does not guarantee that Lanczos will detect
nonpositive curvature in $\bH$ whenever present, because $b$ could be
orthogonal to the subspace corresponding to the nonpositive curvature,
so the Lanczos subspace never intersects with the subspace of negative
curvature. When $b$ is chosen randomly from a uniform distribution
over the unit ball, however, we can certify the performance of
Lanczos, as we have shown in Lemma\refer{~\ref{lemma:randlanczos}
based on Lemma~\ref{lem:KW4.2} above}.  We can exploit
Theorem~\ref{th:lcg} to obtain the same performance for CG, as stated
in Theorem~\ref{theo:randcg}. We restate this result as a corollary,
and prove it now.
\begin{corollary} \label{cor:cgnegcurv}
Let $b$ be distributed uniformly on the unit ball and $H$ be a 
symmetric $n$-by-$n$ matrix, with $\| H \| \le M$. 
Given $\delta \in [0,1)$, define
\begin{equation} \label{eq:lanbound}
\JBB := \min \left\{n, 1+\left\lceil \frac{\ln(2.75n/\delta^2)}{2}
\sqrt{\frac{M}{\epsilon}} \right\rceil \right\}.
\end{equation}
Consider applying Algorithm~\ref{alg:cg} with $\bH
:= H + \frac12 \epsilon I$ and $g=-b$. Then, the following properties hold:
\begin{itemize}
\item[(a)] If $\lambdamin(H) < -\eps$, then with probability at least
  $1-\delta$, there is some index $j \le \JBB$ such that
  Algorithm~\ref{alg:cg} terminates with a direction $p_j$ such that
  $p_j^\top H p_j \le -(\eps/2) \|p_j\|^2$.
\item[(b)] if Algorithm~\ref{alg:cg} runs for $\JBB$
  iterations without terminating, then with probability at least
  $1-\delta$, we have that $\lambdamin(H) \ge -\eps$.
\end{itemize}
\end{corollary}
\begin{proof}
\refer{We will again exploit the invariance of the Krylov subspaces to 
linear shifts given by~\eqref{eq:shiftinvariance}.} This allows us 
to make inferences about the behavior of Algorithm~\ref{alg:cg} 
applied to $\bH$ from the behavior of the
Lanczos method applied to $H$, which has been described in
Lemma~\ref{lemma:randlanczos}.

Suppose first that $\lambdamin(H) < - \eps$. By
Lemma~\ref{lemma:randlanczos}, we know that with probability at least
$1-\delta$, the Lanczos procedure returns a vector $v$ such that
$\|v\|=1$ and $v^\top Hv \le -(\eps/2)$ after at most $\JBB$ iterations.
Thus, for some $j \le \JBB$, we have $v \in \cK_j(b,H) =
\cK_j(b,\bH)$, and moreover $v^\top \bH v \le 0$ by definition of $\bH$,
so the Krylov subspace $\cK_j(b,\bH)$ contains directions of
nonpositive curvature, for some $j \le \JBB$. It then follows from
Theorem~\ref{th:lcg} that $p_j^\top \bH p_j \le 0$ for some $j \le
\JBB$. To summarize: If $\lambdamin(H) < - \eps$, then with
probability $1-\delta$, Algorithm~\ref{alg:cg} applied to $\bH$ and
$g=-b$ will terminate with some $p_j$ such that $p_j^\top \bH p_j \le
0$ for some $j$ with $j \le \JBB$. The proof of (a) is complete. 

Suppose now that Algorithm~\ref{alg:cg} applied to $\bH$ and $g=-b$
runs for $\JBB$ iterations without terminating, that is $p_j^\top \bH
p_j >0$ for $j=0,1,\dotsc,\JBB$. It follows from the logic of
Theorem~\ref{th:lcg} that $\cK_{\JBB}(b,\bH)$ contains no directions
of nonpositive curvature for $\bH$.  Equivalently, there is no
direction of curvature less than $-\eps/2$ for $H$ in
$\cK_{\JBB}(b,H)$.  By Lemma~\ref{lemma:randlanczos}, this certifies
with probability at least $1-\delta$ that $\lambdamin(H) \ge -\eps$,
establishing (b).
\end{proof}


\subsection{\refer{Randomized Lanczos with Internal Estimation of a Bound on $\|H\|$}}
\label{subapp:lacgwithoutM}

The methods discussed in Section~\ref{subapp:lacgwithM} assume
knowledge of an upper bound on the considered matrix, denoted by $M$.
When no such bound is known, we show here that it is possible to
estimate it within the Lanczos procedure. Algorithm~\ref{alg:lanadapt}
details the method; we show that it can be used as an instance of
Procedure~\ref{alg:meo} satisfying Assumption~\ref{assum:wccmeo} when
the optional parameter $M$ is not provided.

Algorithm~\ref{alg:lanadapt} consists in applying the Lanczos method
on $H$ starting with a random vector $b$.  We first run Lanczos for
$j_M$ iterations, where $j_M$ does not depend on any estimate on the
minimum or maximum eigenvalue and instead targets a fixed
accuracy.
After this initial phase of $j_M$ iterations, we have
approximations of the extreme eigenvalues $\ximax(H, b, j_M)$ and
$\ximin(H, b, j_M)$ from \eqref{eq:eigvalsKk}. An estimate $M$ of
$\|H\|$ is then given by:
\begin{equation} \label{eq:estimateMLanczoskM}
	M = 2 \max \left\{ |\ximax(H, b, j_M)|, 
	|\ximin(H, b, j_M)| \right\} .
\end{equation}
We show below that $\|H\| \leq M \le 2\|H\|$, with high
probability. This value can then be used together with a tolerance
$\epsilon$ to define a new iteration limit for the Lanczos
method. After this new iteration limit is reached, we can either
produce a direction of curvature at most $-\epsilon/2$, or certify
with high probability that $\lambda_{\min}(H) \succeq -\epsilon I$ ---
the desired outcomes for Procedure~\ref{alg:meo}.

\begin{algorithm}[ht!]
\caption{Lanczos Method with Upper Bound Estimation}
\label{alg:lanadapt}
\begin{algorithmic}
  \STATE \emph{Inputs:} Symmetric matrix $H \in \R^{n \times n}$, 
  tolerance $\epsilon>0$.
  \STATE \emph{Internal parameters:} probability $\delta \in [0, 1)$, 
  vector $b$ uniformly distributed on the unit sphere.
  \STATE \emph{Outputs:} Estimate $\lambda$ of $\lambdamin(H)$
  such that $\lambda \le - \epsilon/2$, and vector $v$
  with $\|v\|=1$ such that $v^\top H v =\lambda$ OR
  certificate that $\lambdamin(H) \ge -\epsilon$. If the
  certificate is output, it is false with probability $\delta$.
  \STATE Set $j_M = \min \left\{n, 1+\ceil*{\frac12 \ln(25n/\delta^2)}\right\}$.
  \STATE Perform $j_M$ iterations of Lanczos starting from $b$.
  \STATE Compute $\ximin(H,b,j_M)$ and $\ximax(H,b,j_M)$, and set $M$ according 
  to~\eqref{eq:estimateMLanczoskM}.
  \STATE Set $j_{\mathrm{total}} = \min \left\{j_M, 1+
  \ceil*{\frac12  \ln(25n/\delta^2)\sqrt{\frac{M}{\epsilon}}} \right\}$. 
  \STATE Perform $\max\{0,j_{\mathrm{total}}-j_M\}$ additional iterations of Lanczos.
  \STATE Compute $\ximin(H,b,j_{\mathrm{total}})$.
  \IF{$\ximin(H, b, j_{\mathrm{total}}) \leq -\epsilon/2$}
  \STATE \emph{Output} $\lambda = \ximin(H, b, j_{\mathrm{total}})$ and a unit vector
  $v$ such that $v^\top H v = \lambda$.
  \ELSE
  \STATE \emph{Output} $\lambda = \ximin(H, b, j_{\mathrm{total}})$ as a certificate that
  $\lambdamin(H) \geq - \epsilon$. This certificate is false with probability $\delta$.
  \ENDIF
\end{algorithmic}
\end{algorithm}

Algorithm~\ref{alg:lanadapt} could be terminated
  earlier, in fewer than $j_{\mathrm{total}}$ iterations, when a
  direction of sufficient negative is encountered. For simplicity, we
  do not consider this feature, but observe that it would not affect
  the guarantees of the method, described in Lemma~\ref{lem:adaptM}
  below.
\begin{lemma} \label{lem:adaptM}
Consider Algorithm~\ref{alg:lanadapt} with input parameters $H$ and
$b$, and internal parameter $\delta$.  The method outputs a value
$\lambda$ such that
\begin{equation} \label{eq:adaptM1}
\lambda \leq \lambdamin(H) + \frac{\epsilon}{2}
\end{equation}
in at most 
\begin{equation} \label{eq:boundktotalUH}
\min \left\{n, 1+\max\left\{\ceil*{\frac12 \ln(25n/\delta^2)}, 
\ceil*{\frac12 \ln(25n/\delta^2)\sqrt{\frac{2\|H\|}{\epsilon}}} \right\} \right\}
\end{equation}
matrix-vector multiplications by $H$, with probability at least $1 -
\delta$.
\end{lemma}
\begin{proof}
  We begin by showing that the first phase of
  Algorithm~\ref{alg:lanadapt} yields an accurate estimate of $\|H\|$
  with high probability. We assume that $\|H\|>0$ as the result is
  trivially true otherwise.  By setting $\delta \leftarrow \delta/3$
  and $\bar\eps = \frac14$ in Lemma~\ref{lem:KW4.2}, we obtain that
  the following inequalities hold, each with probability at least
  $1-\delta/3$:
\begin{subequations} \label{eq:ba9}
  \begin{align}
    \label{eq:ba9.1}
    \ximax (H,b,j_M) & \ge \lambdamax(H) - \tfrac14 (\lambdamax(H)-\lambdamin(H)), \\
    \label{eq:ba9.2}
\ximin (H,b,j_M) & \le \lambdamin(H) + \tfrac14 (\lambdamax(H)-\lambdamin(H)).
  \end{align}
\end{subequations}
We consider the various possibilities for $\lambdamin(H)$ and
$\lambdamax(H)$ separately, showing in each case that $M$ defined by
\eqref{eq:estimateMLanczoskM} has $\|H\| \le M \le 2 \|H\|$.
\begin{itemize}
  \item When $\lambdamax(H) \ge \lambdamin(H) \ge 0$, we have
    $\ximax(H,b,j_M) \ge \frac34 \lambdamin(H)$ and $0 \le
    \ximin(H,b,j_M) \le \ximax(H,b,j_M)$, so that
    \begin{align*}
    M &= 2 \ximax(H,b,j_M) \ge \frac32 \lambdamax(H) = \frac32\|H\|, \\
    M &= 2 \ximax(H,b,j_M) \le 2 \lambdamax(H) = 2 \|H\|,
    \end{align*}
    as required.
    \item When $\lambdamin(H) \le \lambdamax(H) \le 0$, we have
      $\ximin(H,b,j_M) \le \frac34 \lambdamin(H) \le 0$ and
      $\ximin(H,b,j_M) \le \ximax(H,b,j_M) \le 0$, so that
      \begin{align*}
      M &= 2 | \ximin(H,b,j_M)| \ge \frac32  | \lambdamin(H)| = \frac32 \|H\|, \\
      M &= 2 | \ximin(H,b,j_M)| \le 2 | \lambdamin(H)| = 2 \|H\|,
      \end{align*}
      as required.
      \item When $\lambdamin(H) \le 0 \le \lambdamax(H)$ and
        $-\lambdamin(H) \le \lambdamax(H)$, we have $\lambdamax(H)-\lambdamin(H) \le 2 \lambdamax(H)$, so from \eqref{eq:ba9.1}, it follows that
        $\ximax(H,b,j_M) \ge \frac12 \lambdamax(H) = \frac12 \|H\|$, and so
        \begin{align*}
          M &\ge 2 \ximax(H,b,j_M) \ge \|H\|, \\
          M & = 2 \max\left\{ |\ximax(H, b, j_M)|, 
	|\ximin(H, b, j_M)| \right\}  \le 2 \max\left\{ |\lambdamax(H)|,|\lambdamin(H)| \right\} = 2 \|H\|,
        \end{align*}
        as required.
        \item When $\lambdamin(H) \le 0 \le \lambdamax(H)$ and
          $-\lambdamin(H) \ge \lambdamax(H)$, we have $\lambdamax(H) - \lambdamin(H) \le -2 \lambdamin(H)$, so from \eqref{eq:ba9.2}, it follows that
          $\ximin(H,b,j_M) \le \frac12 \lambdamin(H) \le 0$, and so
          \begin{align*}
            M &\ge 2 |\ximin(H,b,j_M)| \ge | \lambdamin(H)| = \|H\|, \\
            M & =2 \max\left\{ |\ximax(H, b, j_M)|, 
	|\ximin(H, b, j_M)| \right\}  \le 2 \max\left\{ |\lambdamax(H)|,|\lambdamin(H)| \right\} = 2 \|H\|,
        \end{align*}
          as required.
  \end{itemize}
Since each of the bounds in \eqref{eq:ba9} holds with probability at
least $1-\delta/3$, both hold with probability at least $1-2\delta/3$,
by a union bound argument. 
 

We finally consider the complete run of Algorithm~\ref{alg:lanadapt},
which requires $j_{\mathrm{total}}$ iterations of Lanczos.  If our
estimate $M$ is accurate, we have by setting $\delta \leftarrow
\delta/3$ and $M \leftarrow 2 \|H\|$ in Lemma~\ref{lemma:randlanczos}
that $\lambda=\ximin(H,b,j_{\mathrm{total}})$ satisfies
\eqref{eq:adaptM1} with probability $1-\delta/3$. By using a union
bound to combine this probability with the probabilities of estimating
$M$ appropriately, we obtain the probability of at least $1-\delta$.

In conclusion, Algorithm~\ref{alg:lanadapt} runs $j_{\mathrm{total}}$
iterations of Lanczos (each requiring one matrix-vector multiplication
by $H$) and terminates correctly with probability at least
$1-\delta$.
\end{proof}

The lemma above shows that Algorithm~\ref{alg:lanadapt} is an instance
of Procedure~\ref{alg:meo} that does not require an a priori bound on
$\|H\|$. Assuming $\|H\| \le U_H$, we further observe that
Algorithm~\ref{alg:lanadapt} satisfies the conditions of Assumption
\ref{assum:wccmeo} with $\Cmeo =
\frac{\ln(25n/\delta^2)}{\sqrt{2}}\sqrt{U_H}$, which is within a
modest constant multiple of the one obtained for the Lanczos method
with knowledge of $\|H\|$ or $U_H$.

\section{Proof of Theorem~\ref{theo:cvCGwhilestronglycvx}} \label{app:ccg}

\begin{proof}
  The proof proceeds by contradiction: If we assume that all
  conditions specified in the statement of the theorem hold, and in
  addition that
\begin{equation} \label{eq:curvcondyKp1}
\frac{(y_{\hIcg +1}-y_i)^\top \bH (y_{\hIcg+1}-y_i)}
     {\|y_{\hIcg+1}-y_i\|^2} \ge \eps, \quad \mbox{for all
       $i=0,1,\dotsc,\hIcg-1$,}
\end{equation}
then we must have 
\begin{equation} \label{eq:cvCGwhilestronglycvx}
\|r_{\hIcg}\| \le \sqrt{T} \tau^{\hIcg/2} \|r_0\|,
\end{equation}
contradicting \eqref{eq:Kitsstronglycvx}. Note that
\begin{equation} \label{eq:casehK}
		\frac{(y_{\hIcg +1}-y_{\hIcg})^\top \bH (y_{\hIcg+1}-y_{\hIcg})}
		{\|y_{\hIcg+1}-y_{\hIcg}\|^2} = 
		\frac{\alpha_{\hIcg} p_{\hIcg}^\top \bH (\alpha_{\hIcg} p_{\hIcg})}
		{\|\alpha_{\hIcg} p_{\hIcg}\|^2} = \frac{p_{\hIcg}^\top \bH p_{\hIcg}}
		{\|p_{\hIcg}\|^2} \ge \eps
\end{equation}
by assumption, therefore we can consider \eqref{eq:curvcondyKp1} to
hold for $i=0,1,\dotsc,\hIcg$. Moreover, recalling the definition
\eqref{eq:quadraticCG} of the quadratic function $q$, we have for any
$i=0,1,\dotsc,\hIcg$ that 
\[
		q(y_{\hIcg+1}) = q(y_i) + \nabla q(y_i)^\top (y_{\hIcg+1}-y_i) + 
		\frac{1}{2}(y_{\hIcg+1}-y_i)^\top \bH (y_{\hIcg+1}-y_i).
                \]
Thus, \eqref{eq:curvcondyKp1} can be reformulated as
\begin{equation} \label{eq:stronglycvxyKp1yj}
		q(y_{\hIcg+1}) \ge q(y_i) + r_i^\top (y_{\hIcg+1}-y_i) 
		+ \frac{\eps}{2}\|y_{\hIcg+1}-y_i\|^2, \quad
                \mbox{for all $i=0,1,\dotsc,\hIcg$,}
\end{equation}
where we used $\nabla q(y_i)=r_i$ and the definitions
\eqref{eq:curvcondyKp1}, \eqref{eq:casehK} of strong convexity along
the directions $y_{\hIcg+1}-y_i$. In the remainder of the proof, and
similarly to \cite[Proof of Proposition
  1]{YCarmon_JCDuchi_OHinder_ASidford_2017b}, we will show that
\eqref{eq:stronglycvxyKp1yj} leads to the contradiction
\eqref{eq:cvCGwhilestronglycvx}, thus proving that
\eqref{eq:yKp1negcurv} holds.

We define the  sequence of functions $\Phi_j$, $j=0,1,\dotsc\hIcg$ as follows:
\[
    \Phi_0(z) := q(y_0) + \frac{\eps}{2}\|z-y_0\|^2,
    \]
    and for $j=0,\dotsc,\hIcg-1$:
\begin{equation} 
    \label{eq:formulaPhijplus1}
\Phi_{j+1}(z)  := \tau\Phi_j(z) + (1-\tau)\left( q(y_j) + r_j^\top
(z-y_j) +\frac{\eps}{2}\|z-y_j\|^2 \right).
\end{equation}
Since each $\Phi_j$ is a quadratic function with Hessian $\eps I$, it
can be written as follows:
\begin{equation} \label{eq:Phijwithvj}
		\Phi_j(z) \; = \; \Phi_j^* + \frac{\eps}{2}\|z-v_j\|^2,
\end{equation}
where $v_j$ is the unique minimizer of $\Phi_j$, and $\Phi_j^* =
\Phi_j(v_j)$ is the minimum value of $\Phi_j$. (Note that $v_0=y_0=0$
and $\Phi^*_0 = q(y_0) = 0$.)

Defining
\begin{equation} \label{eq:defpsi}
\psi(y) := q(y_0) - q(y) + \frac{\eps}{2}\|y-y_0\|^2 = \Phi_0(y) - q(y),
\end{equation}
we give a short inductive argument to establish that 
\begin{equation} \label{eq:PhiYJp1}
\Phi_j(y_{\hIcg+1}) \le q(y_{\hIcg+1}) + \tau^j \psi(y_{\hIcg+1}),
\quad j=0,1,\dotsc,\hIcg.
\end{equation}
For $j=0$, \eqref{eq:PhiYJp1} holds because $\Phi_0(y) = q(y) +
\psi(y)$ by definition. Assuming that \eqref{eq:PhiYJp1} holds for
some index $j \ge 0$, we find by first applying
\eqref{eq:formulaPhijplus1} (with $z=y_{\hIcg+1}$) and then
\eqref{eq:stronglycvxyKp1yj} (with $i=j$) that
\begin{eqnarray*}
\Phi_{j+1}(y_{\hIcg+1}) &= &\tau\Phi_j(y_{\hIcg+1}) + (1-\tau)\left( q(y_j) + 
r_j^\top (y_{\hIcg+1}-y_j) +\frac{\eps}{2}\|y_{\hIcg+1}-y_j\|^2 \right) \\
&\le &\tau\Phi_j(y_{\hIcg+1}) + (1-\tau) q(y_{\hIcg+1}).
\end{eqnarray*}
Thus, we have
\begin{align*}
\Phi_{j+1}(y_{\hIcg+1}) &\le \tau\Phi_j(y_{\hIcg+1}) + (1-\tau) q(y_{\hIcg+1}) \\
&\le \tau q(y_{\hIcg+1}) + \tau^{j+1} \psi(y_{\hIcg+1}) + 
(1-\tau) q(y_{\hIcg+1}) \quad \mbox{from \eqref{eq:PhiYJp1}} \\
&= q(y_{\hIcg+1}) + \tau^{j+1} \psi(y_{\hIcg+1}),
\end{align*}
which proves \eqref{eq:PhiYJp1} for $j+1$, and thus completes the
inductive argument.

We next prove another technical fact about the relationship between
$q(y_j)$ and $\Phi_j^*$, namely,
\begin{equation} \label{eq:468}
  q(y_j) \le \Phi_j^*, \quad j=0,1,\dotsc, \hIcg.
\end{equation}
We establish this result by an inductive argument that is quite
lengthy and technical; we note the termination of this phase of the
proof clearly below.

This result trivially holds (with equality) at $j=0$.  Supposing that
it holds for some $j =0,1,\dotsc, \hIcg-1$, we will prove that it also
holds for $j+1$.

By using Properties 7 and 8 of Lemma~\ref{lemma:propertiesCGoneiter},
\refer{and also $\|\bH r_j\| \le (M+2\eps)\|r_j\|$,} we have
\begin{equation*} 
	q(y_{j+1}) = q(y_j) - \frac{\|r_j\|^4}{2\, p_j^\top \bH p_j} \le 
	q(y_j) - \frac{\|r_j\|^4}{2\, r_j^\top \bH r_j} \le 
	q(y_j) - \frac{\|r_j\|^2}{2(M+2\eps)}.
\end{equation*}
It follows from induction hypothesis $q(y_j) \le \Phi_j^*$ that 
\begin{align} \label{eq:qyjp1Phijstar}
q(y_{j+1}) \le q(y_j) - \frac{\|r_j\|^2}{2(M+2\eps)}
&= \tau q(y_j) + (1-\tau) q(y_j) -\frac{\|r_j\|^2}{2(M+2\eps)} \nonumber \\
&\le \tau \Phi_j^* + (1-\tau) q(y_j) -\frac{\|r_j\|^2}{2(M+2\eps)}.
\end{align}
By taking the derivative in \eqref{eq:formulaPhijplus1}, and using
\eqref{eq:Phijwithvj}, we obtain
\begin{align*}
\nabla \Phi_{j+1}(z) &= \tau \nabla \Phi_j(z) + (1-\tau)\left[r_j + \eps(z-y_j)\right] \\
\Rightarrow \;\; \eps(z-v_{j+1})&=  \eps \tau (z-v_j) + (1-\tau) \left(r_j + \eps (z-y_j) \right).
\end{align*}
By rearranging the above relation (and noting  that the $z$ terms cancel out), we obtain:
\begin{equation} \label{eq:vjp1minPhijp1}
	v_{j+1}  \; = \; \tau v_j - \frac{1-\tau}{\eps}r_j + (1-\tau) y_j.
\end{equation}
It follows from this expression together with Properties 1 and 2 of
Lemma~\ref{lemma:propertiesCGoneiter} that 
\begin{equation} \label{eq:vjinKj}
v_j \in \lspan\left\{p_0,p_1,\dotsc,p_{j-1}\right\}, \quad j=1,2,\dotsc,\hIcg.
\end{equation}
(The result holds for $j=1$, from \eqref{eq:vjp1minPhijp1} we have
$v_1 \in \lspan \{v_0, r_0, y_0\} = \lspan \{r_0\} = \lspan \{p_0 \}$,
and an induction based on \eqref{eq:vjp1minPhijp1} can be used to
establish \eqref{eq:vjinKj} for the other values of $j$.)  By
combining the expressions~\eqref{eq:Phijwithvj} for $\Phi_j$,
$\Phi_{j+1}$ with the recurrence formula~\eqref{eq:formulaPhijplus1}
for $\Phi_{j+1}$, we obtain
\begin{align*} 
		\Phi_{j+1}^* + \frac{\eps}{2}\|y_j-v_{j+1}\|^2 &= \Phi_{j+1}(y_j) \\
		&= \tau\Phi_j(y_j) + (1-\tau) q(y_j) \\
                &= \tau\left(\Phi_j^* + 
		\frac{\eps}{2}\|y_j-v_j\|^2\right) + (1-\tau) q(y_j)
\end{align*}
and therefore
\begin{equation} \label{eq:doubleexpressionPhijp1}
  \Phi_{j+1}^* = \tau \left(\Phi_j^* +
  \frac{\eps}{2}\|y_j-v_j\|^2\right) + (1-\tau) q(y_j) -
  \frac{\eps}{2}\|y_j-v_{j+1}\|^2.
\end{equation}
On the other hand, we have by~\eqref{eq:vjp1minPhijp1} that
\begin{eqnarray} \label{eq:normylminusvlp1}
		\|y_j-v_{j+1}\|^2 &= &\left\| \tau (y_j-v_j) 
		+\frac{1-\tau}{\eps}r_j \right\|^2 \nonumber \\
		&= &(\tau^2 \|y_j-v_j\|^2 + \frac{(1-\tau)^2}{\eps^2}\|r_j\|^2 
		+ \frac{2}{\eps}(1-\tau)\tau r_j^\top (y_j-v_j) \nonumber \\
		&= & \tau^2 \|y_j-v_j\|^2 + \frac{(1-\tau)^2}{\eps^2}\|r_j\|^2,
\end{eqnarray}
where the last relation comes from $r_j \perp
\lspan\{p_0,\dotsc,p_{j-1}\}$ (Property 2 of
Lemma~\ref{lemma:propertiesCGoneiter}) and \eqref{eq:vjinKj} in the
case $j \ge 1$, and immediately in the case $j=0$, since $y_0=v_0=0$.
By combining \eqref{eq:doubleexpressionPhijp1}
and~\eqref{eq:normylminusvlp1}, we arrive at:
\begin{align} \label{eq:phijp1star}
		\Phi_{j+1}^* &= \tau \left(\Phi_j^* + 
		\frac{\eps}{2}\|y_j-v_j\|^2\right) + (1-\tau) q(y_j) 
		- \frac{\eps}{2}\|y_j-v_{j+1}\|^2 \nonumber \\
		&= \tau \left(\Phi_j^* + 
		\frac{\eps}{2}\|y_j-v_j\|^2\right) + (1-\tau) q(y_j) 
		-\frac{\eps}{2}\tau^2 \|y_j-v_j\|^2 - \frac{(1-\tau)^2}{2\eps} \|r_j\|^2 
		\nonumber \\
		&= \tau \Phi_j^* + \frac{\eps}{2}\left[ \tau 
		-\tau^2\right]\|y_j-v_j\|^2 + (1-\tau) q(y_j) 
		- \frac{(1-\tau)^2}{2\eps} \|r_j\|^2 \nonumber \\ 
		&= \tau \Phi_j^* + 
		\frac{\eps}{2}(1-\tau)\tau \|y_j-v_j\|^2 + (1-\tau) q(y_j) 
		- \frac{(1-\tau)^2}{2\eps} \|r_j\|^2 \nonumber \\
		&\ge \tau \Phi_j^* + (1-\tau) q(y_j) 
		- \frac{(1-\tau)^2}{2\eps} \|r_j\|^2 \nonumber \\
                &\ge q(y_{j+1}) + \frac{1}{2(M+2\eps)}\|r_j\|^2 
		- \frac{(1-\tau)^2}{2\eps} \|r_j\|^2.
\end{align}
where the last inequality comes from~\eqref{eq:qyjp1Phijstar}.  By
using the definitions of $\tau$ and $\kappa$ in
Algorithm~\ref{alg:ccg}, we have
\[
		\frac{(1-\tau)^2}{2\eps} = \frac{1}{2\eps(\sqrt{\kappa}+1)^2} \le 
		\frac{1}{2\eps\kappa} = \frac{1}{2(M+2\eps)}.
\]
It therefore follows from \eqref{eq:phijp1star} that $q(y_{j+1}) \le
\Phi_{j+1}^*$. At this point, we have shown that when $q(y_j) \le
\Phi_j^*$ for $j=0,1,\dotsc,\hIcg-1$, it follows that $q(y_{j+1}) \le
\Phi_{j+1}^*$, establishing the inductive step. As a result, our
proof of \eqref{eq:468} is complete.

By substituting $j=\hIcg$ into \eqref{eq:468}, we obtain
$q(y_{\hIcg}) \le \Phi_{\hIcg}^*$,
which in combination with \eqref{eq:PhiYJp1} with $j=\hIcg$, and the
definition \eqref{eq:Phijwithvj}, yields
\begin{equation} \label{eq:desiredecreaseformula}
q(y_{\hIcg}) -q(y_{\hIcg+1}) \le \Phi_{\hIcg}^* - q(y_{\hIcg+1}) 
\le \Phi_{\hIcg}(y_{\hIcg+1}) - q(y_{\hIcg+1})
\le \tau^{\hIcg} \psi(y_{\hIcg+1}).
\end{equation}
By substitution from \eqref{eq:defpsi}, we obtain
\begin{equation} \label{eq:funcdecreasecontrad}
q(y_{\hIcg}) - q(y_{\hIcg+1}) \; \le \; \tau^{\hIcg}
\left(q(y_0)-q(y_{\hIcg+1}) +
\frac{\eps}{2}\|y_0-y_{\hIcg+1}\|^2\right).
\end{equation}

We now depart from the analysis
of~\cite{YCarmon_JCDuchi_OHinder_ASidford_2017b}, and complete the
proof of this result by expressing \eqref{eq:funcdecreasecontrad} in
terms of residual norms.  On the left-hand side, we have
\begin{align*}
  q(y_{\hIcg}) - q(y_{\hIcg+1})
		&= r_{\hIcg+1}^\top (y_{\hIcg}-y_{\hIcg+1}) + 
		\frac{1}{2}(y_{\hIcg}-y_{\hIcg+1})^\top \bH (y_{\hIcg}-y_{\hIcg+1})  \\
		&= \frac{1}{2}(y_{\hIcg}-y_{\hIcg+1})^\top \bH (y_{\hIcg}-y_{\hIcg+1})
\end{align*}
because $r_{\hIcg+1}^\top (y_{\hIcg}-y_{\hIcg+1}) = r_{\hIcg+1}^\top
(\alpha_{\hIcg} p_{\hIcg}) = 0$ by
Lemma~\ref{lemma:propertiesCGoneiter}, Property 2.
We thus have from \eqref{eq:casehK} that
\begin{alignat}{2}
  \nonumber
  q(y_{\hIcg}) - q(y_{\hIcg+1}) &\ge \frac{\eps}{2}\|y_{\hIcg}-y_{\hIcg+1}\|^2  &&  \\
  \nonumber
  &\refer{= \frac{\eps}{2} \|\alpha_{\hIcg}p_{\hIcg}\|^2} \\
  \nonumber
  &\refer{\ge \frac{\eps}{2 (M+2\eps)^2}\|\bH(\alpha_{\hIcg}p_{\hIcg})\|^2} 
  && \mbox{(since $\|\bH p_{\hIcg}\| \le (M+2\eps)\|p_{\hIcg}\|$)} \\
  \nonumber
  &= \frac{\eps}{2 (M+2\eps)^2} \|\bH(y_{\hIcg}-y_{\hIcg+1})\|^2   \\
  \nonumber
  &= \frac{\eps}{2 (M+2\eps)^2} \|r_{\hIcg} - r_{\hIcg+1}\|^2  && \mbox{(since $r_j=g+\bH y_j$)} \\
  \nonumber
		&= \frac{\eps}{2 (M+2\eps)^2} (\|r_{\hIcg}\|^2 + \|r_{\hIcg+1}\|^2) & \quad & \mbox{(by Lemma~\ref{lemma:propertiesCGoneiter}, Property 2)} \\
  \label{eq:funcdeccontradLHS}
		& \ge \frac{\eps}{2 (M+2\eps)^2} \|r_{\hIcg}\|^2,
\end{alignat}

On the right-hand side of~\eqref{eq:funcdecreasecontrad}, because of
the strong convexity condition \eqref{eq:stronglycvxyKp1yj} at $i=0$,
we have
\begin{align*}
q(y_0)-q(y_{\hIcg+1}) + \frac{\eps}{2}\|y_0-y_{\hIcg+1}\|^2  & \le 
-\nabla q(y_0)^\top (y_{\hIcg+1}-y_0) \\
&= -r_0^\top (y_{\hIcg+1}-y_0)  \le \|r_0\|\|y_{\hIcg+1}-y_0\|.
\end{align*}
Moreover, we have 
\[
		\|y_{\hIcg+1}-y_0\| \; = \;
		\left\|\sum_{i=0}^{\hIcg} \alpha_i p_i \right\| 
		\; \le \; \sum_{i=0}^{\hIcg} \alpha_i \|p_i\| 
		\; = \; \sum_{i=0}^{\hIcg} \frac{\|r_i\|^2}{p_i^\top \bH p_i}\|p_i\|,
\]
where the last relation follows from the definition of $\alpha_i$. By
combining these last two bounds, and using Property 3 of
Lemma~\ref{lemma:propertiesCGoneiter}, we obtain
\begin{equation} \label{eq:wd12}
		q(y_0)-q(y_{\hIcg+1}) + \frac{\eps}{2}\|y_0-y_{\hIcg+1}\|^2 \; \le \; 
		\|r_0\| \sum_{i=0}^{\hIcg} \|r_i\| \frac{\|p_i\|^2}{p_i^\top \bH p_i} \;
		\le \; \|r_0\| \frac{1}{\eps} \sum_{i=0}^{\hIcg} \|r_i\|,
\end{equation}
where the last inequality holds $p_j^\top \bH p_j \ge \eps \|p_j\|^2$
for $j=0,1,\dotsc,\hIcg$ by assumption.

To bound the sum in \eqref{eq:wd12}, we recall that since
Algorithm~\ref{alg:ccg} did not terminate until iteration $\hIcg$,
\refer{the residual norms $\|r_i \|$ at all iterations
  $i=0,1,\dotsc,\hIcg-1$ must have decreased at the expected
  convergence rate. In particularly, we have $\|r_i\| \le \sqrt{T}
  \tau^{i/2} \|r_0\|$ for the {\em possibly smaller} versions of
  $\sqrt{T}$ and $\tau$ that prevailed at iteration $i$, so certainly
  $\|r_i\| \le \sqrt{T} \tau^{i/2} \|r_0\|$ for the final values of
  these parameters.} Thus for $i=0,1,\dotsc,\hIcg-1$, we have
\[
\|r_i\| \le \sqrt{T} \tau^{i/2} \|r_0\| \le \tau^{(i-\hIcg)/2}\|r_{\hIcg}\|,
\]
where we used $\|r_{\hIcg}\| \ge \sqrt{T} \tau^{\hIcg/2} \|r_0\|$
(from \eqref{eq:Kitsstronglycvx}) for the last inequality. Observing
that this bound also holds (trivially) for $i=\hIcg$, we obtain by
substituting in \eqref{eq:wd12} that
\begin{align}
  \nonumber
q(y_0)-q(y_{\hIcg+1}) + \frac{\eps}{2}\|y_0-y_{\hIcg+1}\|^2
&\le \|r_0\| \frac{1}{\eps} \sum_{i=0}^{\hIcg} \tau^{(i-\hIcg)/2} \|r_{\hIcg}\| \\
\nonumber
&\le \|r_0\| \frac{\tau^{-\hIcg/2}}{\eps}\|r_{\hIcg}\|\,\sum_{i=0}^{\hIcg} (\sqrt{\tau})^{i} \\
\label{eq:funcdeccontradRHS}
& \le \|r_0\| \frac{\tau^{-\hIcg/2}}{\eps}\|r_{\hIcg}\| \frac{1}{1-\sqrt{\tau}}.
\end{align}
Applying successively \eqref{eq:funcdeccontradLHS},
\eqref{eq:funcdecreasecontrad} and \eqref{eq:funcdeccontradRHS}
finally yields:
\begin{align*}
		\frac{\eps}{2 (M+2\eps)^2} \|r_{\hIcg}\|^2 &\le 
		 q(y_{\hIcg}) - q(y_{\hIcg+1}) \\
		 &\le \tau^{\hIcg} \left( q(y_0)-q(y_{\hIcg+1}) + \frac{\eps}{2}\|y_0-y_{\hIcg+1}\|^2 \right) \\
		 &\le \tau^{\hIcg} \|r_0\|\|r_{\hIcg}\| 
		 \frac{\tau^{-\hIcg/2}}{\eps}\frac{1}{1-\sqrt{\tau}}.
\end{align*}
After rearranging this inequality and dividing by $\|r_{\hIcg}\| > 0$,
we obtain
\begin{equation} \label{eq:finalcontrad}
		\|r_{\hIcg}\| \le \frac{2 (M+2\eps)^2}{\eps^2}\frac{\tau^{\hIcg/2}}{1-\sqrt{\tau}} \|r_0\|
		= \sqrt{T} \tau^{\hIcg/2}\|r_0\|.
\end{equation}
We have thus established \eqref{eq:cvCGwhilestronglycvx} which, as we
noted earlier, contradicts \eqref{eq:Kitsstronglycvx}. Thus
\eqref{eq:curvcondyKp1} cannot be true, so we have established
\eqref{eq:yKp1negcurv}, as required.
\end{proof}

\end{document}